\documentclass{amsart}
\usepackage{amssymb, amsbsy, amsthm, amsmath, amstext, amsopn, verbatim}
\usepackage[all]{xy}
\usepackage{amsfonts}
\usepackage{amscd}
\usepackage{mathrsfs}
\usepackage{hyperref}
\usepackage{verbatim}
\usepackage{array,multirow,booktabs,longtable}
\RequirePackage{tikz-cd}
\usepackage{tikz}

\newtheorem{thm}{Theorem} [section]

\newtheorem{lem}[thm]{Lemma}

\newtheorem{cor}[thm]{Corollary}
\newtheorem{prop}[thm]{Proposition}

\newtheorem*{basic assumption}{Basic Assumption}

\theoremstyle{definition}

\newtheorem*{principal example}{Main Example}

\newtheorem{defn}[thm]{Definition}

\newtheorem{example}[thm]{Example}

\theoremstyle{remark}

\newtheorem{rem}[thm]{Remark}
                         
\newtheorem{claim}[thm]{Claim}

\begin{document}

\numberwithin{equation}{section}

\newcommand{\hs}{\mbox{\hspace{.4em}}}
\newcommand{\bd}{\begin{displaymath}}
\newcommand{\ed}{\end{displaymath}}
\newcommand{\bcd}{\begin{CD}}
\newcommand{\ecd}{\end{CD}}

\newcommand{\proj}{\operatorname{Proj}}
\newcommand{\bproj}{\underline{\operatorname{Proj}}}
\newcommand{\spec}{\operatorname{Spec}}
\newcommand{\bspec}{\underline{\operatorname{Spec}}}
\newcommand{\pline}{{\mathbf P} ^1}
\newcommand{\pplane}{{\mathbf P}^2}
\newcommand{\coker}{{\operatorname{coker}}}
\newcommand{\ldb}{[[}
\newcommand{\rdb}{]]}

\newcommand{\Sym}{\operatorname{Sym}^{\bullet}}
\newcommand{\Symp}{\operatorname{Sym}}
\newcommand{\Pic}{\operatorname{Pic}}
\newcommand{\AAut}{\operatorname{Aut}}
\newcommand{\PAut}{\operatorname{PAut}}

\newcommand{\too}{\twoheadrightarrow}
\newcommand{\C}{{\mathbb C}}
\newcommand{\cA}{{\mathcal A}}
\newcommand{\cS}{{\mathcal S}}
\newcommand{\cV}{{\mathcal V}}
\newcommand{\cM}{{\mathcal M}}
\newcommand{\bA}{{\mathbf A}}
\newcommand{\cB}{{\mathcal B}}
\newcommand{\cC}{{\mathcal C}}
\newcommand{\cD}{{\mathcal D}}
\newcommand{\D}{{\mathcal D}}
\newcommand{\boldc}{{\mathbf C}}
\newcommand{\cE}{{\mathcal E}}
\newcommand{\cF}{{\mathcal F}}
\newcommand{\cG}{{\mathcal G}}
\newcommand{\G}{{\mathbf G}}

\newcommand{\bH}{{\mathbf H}}
\newcommand{\cH}{{\mathcal H}}
\newcommand{\cI}{{\mathcal I}}
\newcommand{\cJ}{{\mathcal J}}
\newcommand{\cK}{{\mathcal K}}
\newcommand{\cL}{{\mathcal L}}
\newcommand{\baL}{{\overline{\mathcal L}}}
\newcommand{\M}{{\mathcal M}}
\newcommand{\bM}{{\mathbf M}}
\newcommand{\bm}{{\mathbf m}}
\newcommand{\cN}{{\mathcal N}}
\newcommand{\theo}{\mathcal{O}}
\newcommand{\cP}{{\mathcal P}}
\newcommand{\cR}{{\mathcal R}}
\newcommand{\boldp}{{\mathbf P}}
\newcommand{\boldq}{{\mathbf Q}}
\newcommand{\bbL}{{\mathbf L}}
\newcommand{\cQ}{{\mathcal Q}}
\newcommand{\cO}{{\mathcal O}}
\newcommand{\Oo}{{\mathcal O}}
\newcommand{\OX}{{\Oo_X}}
\newcommand{\OY}{{\Oo_Y}}
\newcommand{\dd}{\mathcal{D}}
\newcommand{\Hamp}{\mathbb{H}^{\perp}}
\newcommand{\Ham}{\mathbb{H}}
\newcommand{\cX}{{\mathcal X}}
\newcommand{\cW}{{\mathcal W}}
\newcommand{\boldz}{{\mathbf Z}}
\newcommand{\cZ}{{\mathcal Z}}
\newcommand{\qgr}{\operatorname{qgr}}
\newcommand{\gr}{\operatorname{gr}}
\newcommand{\coh}{\operatorname{coh}}
\newcommand{\End}{\operatorname{End}}
\newcommand{\Hom}{\operatorname{Hom}}
\newcommand{\IndCoh}{\operatorname{IndCoh}}
\newcommand{\sHom}{\mathcal{H}om}
\newcommand{\sEnd}{\mathcal{E}nd}
\newcommand{\uHom}{\underline{\operatorname{Hom}}}
\newcommand{\uHomY}{\uHom_{\OY}}
\newcommand{\uHomX}{\uHom_{\OX}}
\newcommand{\Ext}{\operatorname{Ext}}
\newcommand{\bExt}{\operatorname{\bf{Ext}}}
\newcommand{\Tor}{\operatorname{Tor}}

\newcommand{\inv}{^{-1}}
\newcommand{\airtilde}{\widetilde{\hspace{.5em}}}
\newcommand{\airhat}{\widehat{\hspace{.5em}}}
\newcommand{\nt}{^{\circ}}
\newcommand{\del}{\partial}

\newcommand{\supp}{\operatorname{supp}}
\newcommand{\GK}{\operatorname{GK-dim}}
\newcommand{\W}{W}
\newcommand{\id}{\operatorname{id}}
\newcommand{\res}{\operatorname{res}}
\newcommand{\lrar}{\leadsto}
\newcommand{\im}{\operatorname{Im}}
\newcommand{\HH}{\operatorname{H}}
\newcommand{\Coh}[1]{#1\text{-}{\mathsf{coh}}}
\newcommand{\QCoh}[1]{#1\text{-}{\mathsf{qcoh}}}
\newcommand{\PCoh}[1]{#1\text{-}{\mathsf{procoh}}}
\newcommand{\Good}[1]{#1\text{-}{\mathsf{good}}}
\newcommand{\QGood}[1]{#1\text{-}{\mathsf{Qgood}}}
\newcommand{\Hol}[1]{#1\text{-}{\mathsf{hol}}}
\newcommand{\Reghol}[1]{#1\text{-}{\mathsf{reghol}}}
\newcommand{\Bun}{\operatorname{Bun}}
\newcommand{\Hilb}{\operatorname{Hilb}}
\newcommand{\pa}{\partial}
\newcommand{\F}{\mathcal{F}}
\newcommand{\nthord}{^{(n)}}
\newcommand{\Aut}{\underline{\operatorname{Aut}}}
\newcommand{\Gr}{\operatorname{\bf Gr}}
\newcommand{\Fr}{\operatorname{Fr}}
\newcommand{\GL}{\operatorname{GL}}
\newcommand{\gl}{\mathfrak{gl}}
\newcommand{\SL}{\operatorname{SL}}
\newcommand{\ff}{\footnote}
\newcommand{\ot}{\otimes}
\newcommand{\Wx}{\mathcal W_{\mathfrak X}}
\newcommand{\gh}{\text{gr}_{\hbar}}
\newcommand{\ig}{\iota_g}
\def\Ext{\operatorname {Ext}}
\def\Hom{\operatorname {Hom}}

\def\bbZ{{\mathbb Z}}
\newcommand{\iso}{{\;\stackrel{_\sim}{\to}\;}}

\newcommand{\nc}{\newcommand}
\newcommand{\on}{\operatorname}
\nc{\cont}{\on{cont}}
\nc{\rmod}{\on{mod}}
\nc{\Mtil}{\widetilde{M}}
\nc{\wb}{\overline}
\nc{\wt}{\widetilde}
\nc{\wh}{\widehat}
\nc{\sm}{\setminus}
\nc{\mc}{\mathcal}
\nc{\mbb}{\mathbb}
\nc{\Mbar}{\wb{M}}
\nc{\Nbar}{\wb{N}}
\nc{\Mhat}{\wh{M}}
\nc{\pihat}{\wh{\pi}}
\nc{\opp}{\mathrm{opp}}
\nc{\phitil}{\wt{\phi}}
\nc{\Qbar}{\wb{Q}}
\nc{\DYX}{\D_{Y\leftarrow X}}
\nc{\DXY}{\D_{X\to Y}}
\nc{\dR}{\stackrel{\bbL}{\underset{\D_X}{\ot}}}
\nc{\Winfi}{\cW_{1+\infty}}
\nc{\K}{{\mc K}}
\nc{\unit}{{\bf \on{unit}}}
\nc{\boxt}{\boxtimes}
\nc{\xarr}{\stackrel{\rightarrow}{x}}
\nc{\Cnatbar}{\overline{C}^{\natural}}
\nc{\oJac}{\overline{\on{Jac}}}
\nc{\gm}{{\mathbf G}_m}
\nc{\Loc}{\on{Loc}}
\nc{\Bm}{\operatorname{Bimod}}
\nc{\lie}{{\mathfrak g}}
\nc{\lb}{{\mathfrak b}}
\nc{\lien}{{\mathfrak n}}
\nc{\E}{\mathcal{E}}
\nc{\Cs}{\mathbb{G}_m}
\nc{\hol}{\mathrm{Hol}}
\nc{\can}{\mathrm{can}}
\newcommand{\idot}{{\:\raisebox{2pt}{\text{\circle*{1.5}}}}}

\nc{\Gm}{{\mathbb G}_m}
\nc{\Gabar}{\wb{\G}_a}
\nc{\Gmbar}{\wb{\G}_m}
\nc{\PD}{{\mathbb P}_{\D}}
\nc{\Pbul}{P_{\bullet}}
\nc{\PDl}{{\mathbb P}_{\D(\lambda)}}
\nc{\PLoc}{\mathsf{MLoc}}
\nc{\Tors}{\on{Tors}}
\nc{\PS}{{\mathsf{PS}}}
\nc{\PB}{{\mathsf{MB}}}
\nc{\Pb}{{\underline{\operatorname{MBun}}}}
\nc{\Ht}{\mathsf{H}}
\nc{\bbH}{\mathbb H}
\nc{\gen}{^\circ}
\nc{\Jac}{\operatorname{Jac}}
\nc{\sP}{\mathsf{P}}
\nc{\otc}{^{\otimes c}}
\nc{\Det}{\mathsf{det}}
\nc{\PL}{\on{ML}}

\nc{\ml}{{\mathcal S}}
\nc{\Xc}{X_{\on{con}}}
\nc{\sgood}{\text{strongly good}}
\nc{\Xs}{X_{\on{strcon}}}
\nc{\resol}{\mathfrak{X}}
\nc{\map}{\mathsf{f}}
\nc{\tor}{\mathrm{tor}}
\nc{\base}{Z}
\nc{\bigvar}{\mathsf{W}}
\nc{\alg}{\mathsf{A}}
\nc{\T}{\mathsf{T}}
\nc{\qcoh}{\on{qcoh}}
\renewcommand{\o}{\otimes}
\nc{\mf}{\mathfrak}
\nc{\NN}{\mathsf{N}}
\nc{\h}{\hbar}
\nc{\ms}{\mathscr}
\nc{\good}{\mathrm{good}}
\newcommand{\LMod}[1]{#1\text{-}{\mathsf{Mod}}}
\newcommand{\Lmod}[1]{#1\text{-}{\mathsf{mod}}}
\nc{\mbf}{\mathbf}
\nc{\ad}{\mathrm{ad}}
\nc{\Rees}{\mathsf{Rees}}
\nc{\Supp}{\mathrm{Supp}}
\nc{\Z}{\mathbb{Z}}
\nc{\N}{\mathbb{N}}
\nc{\ann}{\mathrm{ann}}
\nc{\Blf}{B_{\mathrm{l.f.}}}
\nc{\bx}{\mathbf{x}}
\nc{\by}{\mathbf{y}}
\nc{\bz}{\mathbf{z}}
\nc{\bw}{\mathbf{w}}
\nc{\Der}{\mathrm{Der}}
\nc{\CatCs}{\ms{C}}
\renewcommand{\mod}{ \ \mathrm{mod} \ }
\nc{\sA}{\mc{A}} 
\nc{\A}{A} 
\nc{\B}{B} 
\nc{\sW}{\mathscr{W}} 
\nc{\rh}{\mathrm{r.h.}}
\nc{\cs}{\mathbb{C}^*}
\nc{\R}{\mathbb{R}}
\nc{\Lie}{\mathrm{Lie}}
\nc{\Lag}{\mathsf{\Lambda}}
\nc{\ds}{\displaystyle}
\nc{\Qcoh}{\mathsf{Qcoh}}
\nc{\WQcoh}{\mathsf{Qcoh}(\cW)}
\nc{\Indf}{\mathsf{Ind}}
\nc{\DR}{\mathsf{DR}}
\nc{\co}{\operatorname{co}}
\nc{\op}{\operatorname{\opp}}
\nc{\Eq}{\mathrm{Eq}}
\renewcommand{\L}{\mathbb{L}}
\newcommand{\gwyn}[1]{\textcolor{brown}{#1}}

\title{A Localization Theorem For Finite W-algebras}

\author{Christopher Dodd and Kobi Kremnizer}
\begin{abstract}
Following the work of Beilinson-Bernstein \cite{key-15} and Kashiwara-Rouquier
\cite{key-7}, we give a geometric interpretation of certain categories
of modules over the finite W-algebra. Along the way, we give a new,
general Hamiltonian reduction formalism for quantizations and we reprove
the Skryabin equivalence.
\end{abstract}

\maketitle

\section{Introduction}

Let $\mathfrak{g}$ be a complex semisimple Lie algebra, and $U(\mathfrak{g})$
its enveloping algebra. The Beilinson-Bernstein localization theorem
\cite{key-15} gives a geometric interpretation of the category of
finitely generated modules over $U(\mathfrak{g})$ with trivial central
character. In particular, this category is equivalent to the category
$Mod^{coh}(D(G/B))$ of coherent $D-$modules on the flag variety
associated to $G$. This result can be explained as follows: there
is a natural map $T^{*}(G/B)\to N$ which is a resolution of singularities.
The normality of the variety $N$ implies that $\Gamma(N,O_{N})=\Gamma(T^{*}(G/B),O_{T^{*}G/B})$.
Further, the ring $U(\mathfrak{g})_{0}$ can be thought of as a quantization
of the nilpotent cone $N$, and the sheaf $D_{G/B}$ can be thought
of as a quantization of the variety $T^{*}(G/B)$. However, the sheaf
$D_{G/B}$ is not local on $T^{*}(G/B)$, only on $G/B$ itself. 

Kashiwara and Rouquier (in \cite{key-7}) give a framework for reformulating
this theorem using a notion of sheaves of asymptotic differential
operators. One can define a sheaf of algebras $D_{h}(G/B)$ on the
variety $T^{*}(G/B)$, which is (in some sense) a quantization. This
sheaf is defined over the power series field $\mathbb{C}((h))$, and
therefore the category of modules over it is not equivalent to a $\mathbb{C}$-linear
category of modules over $U(\mathfrak{g})$. However, this can be
corrected by considering the $\mathbb{C}^{*}$-action on $T^{*}(G/B)$
given by dilating the fibres. In particular, there is a notion of
$\mathbb{C}^{*}$-equivariant $D_{h}(G/B)$-module for which the equivalence
$Mod^{coh,\mathbb{C}^{*}}(D_{h}(G/B))\tilde{\to}Mod^{f.g.}(U(\mathfrak{g})_{0})$
holds. 

Our goal in this paper is to give a version of this theorem for the
finite W-algebras. We give the precise definition of these objects
below. For now, we simply note that given a nilpotent element $e\in N$,
there is subvariety $S_{e}\subseteq N$ called the transverse slice
to the $G$-orbit at $e$. This variety admits a natural $\mathbb{C}^{*}$-action
that contracts it to $e$. Then there is a filtered, noncommutative
algebra $U(\mathfrak{g},e)_{0}$ (the finite W-algebra at $e$ with
trivial central character) such that $gr(U(\mathfrak{g},e)_{0})\tilde{=}O(S_{e})$,
where the grading on $S_{e}$ is given by the aforementioned $\mathbb{C}^{*}$-action. 

There is a resolution of singularities $\tilde{S}_{e}\to S_{e}$ where
$\tilde{S}_{e}$ is the (set-theoretic) inverse image of $S_{e}$
under the map $T^{*}(G/B)\to N$. In addition, there is a $\mathbb{C}^{*}$-action
on $T^{*}(G/B)$ that preserves $\tilde{S}_{e}$ and for which the
resolution $\tilde{S}_{e}\to S_{e}$ is equivariant. Then, our main
theorem gives a sheaf of algebras on $\tilde{S}_{e}$ called $D_{h}(0,\chi)$,
which is (in a sense) a quantization of $\tilde{S}_{e}$ and for which
there is the equivalence $Mod^{coh,\mathbb{C}^{*}}(D_{h}(0,\chi))\tilde{\to}Mod^{f.g.}(U(\mathfrak{g},e)_{0})$.
Our proof relies on the fact that the variety $\tilde{S}_{e}$ is
not only a subvariety of $T^{*}(G/B)$ but can also be obtained via
the procedure of ``Hamiltonian reduction.'' We can then obtain the
sheaf $D_{h}(0,\chi)$ via the procedure of Hamiltonian reduction
of the sheaf $D_{h}(G/B)$. The proof of the result follows the same
lines as the proof of the classical Beilinson-Bernstein theorem. 

In the body of the paper, we prove all of the results outlined above,
in a slightly more general form. In particular, we work with categories
of modules over any anti-dominant central character, not just the
trivial one. Further, we give several applications to the theory of
W-algebras, including reproving the well-known Skryabin equivalence.
Beyond this, the results of this paper have been cited, e.g., in \cite{key-22},
and the paper \cite{key-20} provides a generalization to the affine
case. Another related reference is \cite{key-16}, chapter 6, where
different presentation of modules over the $W$-algebra with a given
central character is given. In that work, the idea is to present coherent
sheaves over the quasi-projective scheme $\tilde{S}_{e}$ as a Serre
quotient of modules over a suitable graded algebra; and then to present
a certain non-commutative deformation of that algebra (called a directed
algebra) so that a suitable Serre quotient of its modules are equivalent
to modules over the $W$-algebra. While there do not seem to be direct
implications in either direction, that work is certainly morally related
to this one. 

\subsection{Acknowledgements}

The authors would like to express their gratitude to an anonymous
referee, whose insightful comments (including a correction of the
key definition \ref{def:twist}) have vastly improved the paper. 

\section{\label{sec:W-algebras-and-quantum}W-algebras and Quantum Hamiltonian
Reduction}

Let $A$ be an associative algebra over $\mathbb{C}$, and let $M$
be a connected affine algebraic group; we set $Lie(M)=\mathfrak{m}$.
We suppose that there is an action of $M$ on $A$ that is algebraic
(i.e., locally finite), and respects the algebra structure. We assume
given an algebra morphism $\rho:U\mathfrak{m}\to A$ such that the
adjoint action of $\mathfrak{m}$ on $A$ (i.e., the action given
by $ad(m)(a)=\rho(m)a-a\rho(m)$ for all $m\in\mathfrak{m},a\in A$)
is the differential of the $M$ action. Let $I\subseteq U\mathfrak{m}$
be a two-sided ideal. Then it is easy to see that $(A/A\rho(I))^{M}$
inherits an algebra structure from $A$, called the quantum Hamiltonian
reduction of $A$ with respect to $I$. If there exists a character
$\chi$ on $\mathfrak{m}$ such that $I=ker(\chi)$ (where we also
use the letter $\chi$ to denote the unique extension of this character
to a character of $U\mathfrak{m}$), then we can describe the algebra
structure on $(A/A\rho(I))^{M}$ via an isomorphism $(A/A\rho(I))^{M}\tilde{\to}End_{A}(A/AI)^{op}$
that takes $u\in(A/A\rho(I))^{M}$ to right multiplication by $u$
in $A/AI$. 

We will now define the finite W-algebra $U(\mathfrak{g},e)$ via the
quantum Hamiltonian reduction procedure. For references on everything
in this section, see \cite{key-3}. We let $e\in\mathfrak{g}$ be
a nonzero nilpotent element. By the Jacobson-Morozov theorem, there
exist $f,h\in\mathfrak{g}$ such that $\{e,f,h\}$ form an $\mathfrak{sl_{2}}$-triple,
and we fix such a triple throughout. Given this, the adjoint action
makes $\mathfrak{g}$ into a finite dimensional $\mathfrak{sl_{2}}$-module,
and we have the corresponding weight decomposition $\mathfrak{g}=\oplus\mathfrak{g}(i)$,
where $\mathfrak{g}(i)=\{x\in\mathfrak{g}|[h,x]=ix\}$. This makes
$\mathfrak{g}$ into a graded Lie algebra. We let $\chi\in\mathfrak{g}^{*}$
be the element associated to $e$ under the isomorphism $\mathfrak{g}\tilde{=}\mathfrak{g}^{*}$
given by the Killing form. We define a skew-symmetric bilinear form
on $\mathfrak{g}(-1)$ via 
\[
<x,y>=\chi([x,y])
\]
which is easily seen to be non-degenerate. Thus, $(\mathfrak{g}(-1),<,>)$
is a symplectic vector space, and we choose $l\subset\mathfrak{g}(-1)$
a Lagrangian subspace. We define $\mathfrak{m}_{l}=l\oplus\bigoplus_{i\leq-2}g(i)$,
a nilpotent Lie algebra such that $\chi|_{\mathfrak{m}_{l}}$ is a
character of $\mathfrak{m}_{l}$. We let $M_{l}$ be the unipotent
connected algebraic subgroup of $G$ such that $Lie(M_{l})=\mathfrak{m}_{l}$.
Then $M_{l}$ acts on $U\mathfrak{g}$ via the adjoint action, and
we let $I\subset U\mathfrak{m}_{l}$ be the kernel of the character
$\chi$. So we see that we are in the setup of a quantum Hamiltonian
reduction (where $A=U\mathfrak{g}$, and $\rho:U\mathfrak{m}_{l}\to U\mathfrak{g}$
is the natural inclusion). 
\begin{defn}
The finite W-algebra associated to $e\in\mathfrak{g}$, denoted $U(\mathfrak{g},e)$,
is the quantum Hamiltonian reduction of $U\mathfrak{g}$ with respect
to $M_{l}$ and the ideal $I\subset U\mathfrak{m}_{l}$. 
\end{defn}

For example, if $e$ is a regular nilpotent element, then $U(\mathfrak{g},e)\tilde{=}Z(U\mathfrak{g})$;
we always have a canonical map $Z(U\mathfrak{g})\to U(\mathfrak{g},e)$
because $Z(U\mathfrak{g})=U(\mathfrak{g})^{G}\subset U(\mathfrak{g})^{M_{l}}$.
In fact, this map is always an isomorphism onto the center of $U(\mathfrak{g},e)$
(as explained in \cite{key-14} section 5, footnote 2). In case $e$
is regular, the map is actually surjective. 

We wish to ``explain'' the finite W-algebra by expressing it as
a quantization of the algebra of functions on the Slodowy slice $S\subset\mathfrak{g}^{*}$,
which is the image under $\mathfrak{g}\tilde{=}\mathfrak{g}^{*}$
of the affine subspace $e+ker(adf)$. 

To make this more precise, we introduce a $\mathbb{C}^{*}$ action
on $\mathfrak{g}$ as follows: our chosen $\mathfrak{sl}_{2}$-triple
gives a homomorphism $\tilde{\gamma}:SL_{2}(\mathbb{C})\to G$, and
we define $\gamma(t)=\tilde{\gamma}\begin{pmatrix}t & 0\\
0 & t^{-1}
\end{pmatrix}$, so that $Ad(\gamma(t)e=t^{2}e$; so we define $\bar{\rho}(t)=t^{-2}Ad(\gamma(t))$,
a $\mathbb{C}^{*}$-action on $\mathfrak{g}$ that stabilizes $S$
and fixes $e$ (in fact, the inverse of this action contracts $S$
to $e$). So, this action induces a grading on $S^{\bullet}\mathfrak{g}=\mathbb{C}[\mathfrak{g}^{*}]$
and $\mathbb{C}[S]$ (where we now think of $S\subset\mathfrak{g}^{*}$by
using the Killing form to identify $\mathfrak{g}$ and $\mathfrak{g}^{*}$,
and transport the $\mathbb{C}^{*}$-action accordingly). This grading
can now be described explicitly as follows: one writes $S^{\bullet}\mathfrak{g}=\bigoplus_{n\geq0}S^{n}\mathfrak{g}$,
the decomposition using the standard grading, and we let $S^{n}\mathfrak{g}(i)=\{x\in S^{n}\mathfrak{g}|[h,x]=ix\}$,
where $h\in\mathfrak{g}$ is as before, and the bracket denotes the
unique extension of the adjoint action of $h$ on $\mathfrak{g}$
to a derivation of $S^{\bullet}\mathfrak{g}$. The grading defined
above is then obtained by setting $S^{\bullet}\mathfrak{g}[n]=span\{S^{j}\mathfrak{g}(i)|i+2j=n\}$
for all $n\in\mathbb{Z}$ (note that negative degrees do in fact occur).
Then the grading on $\mathbb{C}[S]$ is the one inherited from $S^{\bullet}\mathfrak{g}$,
and it is easy to see that $\mathbb{C}[S]$ has only positive degrees
under this grading. Now, we define the Kazhdan filtration on $U\mathfrak{g}$
by first setting $U_{n}\mathfrak{g}(i)=\{x\in U_{n}\mathfrak{g}|[h,x]=ix\}$,
where $U\mathfrak{g}=\cup U_{n}\mathfrak{g}$ is the usual (PBW) filtration,
and the bracket is just the bracket in $U\mathfrak{g}$, and then
defining $F_{n}U\mathfrak{g}=span\{x\in U\mathfrak{g}_{j}(i)|i+2j\leq n\}$
for all $n\in\mathbb{Z}$. Then an easy application of the PBW theorem
shows that, considering $U\mathfrak{g}$ and $S^{\bullet}\mathfrak{g}$
with the above filtration and grading, $Gr(U\mathfrak{g})=S^{\bullet}\mathfrak{g}$.
If we let $U(\mathfrak{g},e)$ have the inherited filtration, then
we have
\begin{thm}
$Gr(U(\mathfrak{g},e))=\mathbb{C}[S]$
\end{thm}

This isomorphism also puts a natural Poisson structure on $\mathbb{C}[S]$,
which is described in \cite{key-3}. 

Because of this theorem, the algebra $U(\mathfrak{g},e)$ is sometimes
referred to as the enveloping algebra of the slice $S$. 

\subsection{Hamiltonian Reduction for Asymptotic Enveloping Algebras}

In this subsection, we explain a version of the above constructions
for so-called asymptotic enveloping algebras- this variant will be
used repeatedly below. To set things up, we consider, as above, $M$,
a connected affine algebraic group; we set $Lie(M)=\mathfrak{m}$. 
\begin{defn}
\label{def:Asymptotic-Env} Let $U_{h}(\mathfrak{m})(0)$ be the algebra
defined as the $h$-completion of the algebra $T^{\bullet}\mathfrak{m}/I$,
where $I$ is the two-sided ideal in the tensor algebra $T^{\bullet}\mathfrak{m}$
over the polynomial ring $\mathbb{C}[h]$ generated by $\{xy-yx-h[x,y]|x,y\in\mathfrak{m}\}$.
Further, $U_{h}(\mathfrak{m})$ will denote $U_{h}(\mathfrak{m})(0)[h^{-1}]$.
The algebra $U_{h}(\mathfrak{m})(0)$ is called the asymptotic enveloping
algebra of $\mathfrak{m}$. 
\end{defn}

By construction one has 
\[
U_{h}(\mathfrak{m})(0)/h\tilde{\to}\mathrm{S}^{\bullet}(\mathfrak{m})
\]
where the object on the right is the symmetric algebra of $\mathfrak{m}$
over $\mathbb{C}$. It is not difficult to see that $U_{h}(\mathfrak{m})(0)$
is an $h$-complete $\mathbb{C}[[h]]$ algebra, which is is noetherian.
In particular, every ideal (either left, right, or two-sided) of $U_{h}(\mathfrak{m})(0)$
is finitely generated and $h$-complete. 

From the adjoint action of the group $M$ on $T^{\bullet}\mathfrak{m}$,
we deduce an action (also called the adjoint action) on $T^{\bullet}\mathfrak{m}/I$;
and therefore an $M$-action on each quotient $(T^{\bullet}\mathfrak{m}/I)/h^{n}$
for which the quotient maps are equivariant. Passing to the inverse
limit, we obtain an $M$ action on $U_{h}(\mathfrak{m})(0)$ that
we refer to as the adjoint action; inverting $h$ also gives an action
on $U_{h}(\mathfrak{m})$. 

We wish to consider the analogue of the Hamiltonian reduction construction
in this situation. To do so, we need the correct analogue of the algebra
$A$, equipped with its $M$-action. For that we give the 
\begin{defn}
\label{def:Ham-Reduc-for-quantized}Let $A_{h}$ be an associative
$\mathbb{C}[[h]]$-algebra, which is $h$-flat and complete with respect
to the $h$-adic topology. Suppose also\footnote{The noetherian condition probably isn't strictly necessary, but it
is always satisfied in this paper and so we assume it for convenience.} that $A_{h}$ is left and right noetherian. Then we say that $A_{h}$
is $M$-equivariant if there is an algebraic action of $M$, by algebra
automorphisms, on each reduction $A_{h}/h^{n}$ so that all the quotient
maps are $M$-equivariant. In this case there is an induced action
of $M$ on $A_{h}$, and we demand that $m\cdot h=h$ for all $m\in M$. 

Suppose we are given a map of algebras $\rho:U_{h}(\mathfrak{m})(0)\to A$,
which takes $h$ to $h$ and is continuous with respect to the $h$-adic
topology. Such a map is called a comoment map; we say that this map
is compatible with the action of $M$ if, upon reduction mod $h^{n}$
for each $n$, we have that the action $a\to\rho(m)a-a\rho(m)$ is
given by $h\cdot d(m)$ (where $d(m)$ denotes the differential of
the given action of $M$). 
\end{defn}

In this set-up, we have the following version of Hamiltonian reduction: 
\begin{prop}
\label{prop:Ham-reduction-is-algebra}Let $J$ be a two-sided ideal
of $U_{h}(\mathfrak{m})(0)$. Suppose that $A/A\cdot J$ is $h$-torsion
free. Then the complete $\mathbb{C}[[h]]$-module $(A/A\cdot J)^{\mathfrak{m}}$
is naturally an algebra so that, if $a,b\in A$ are elements whose
images in $A/A\cdot J$ are denoted $\bar{a}$ and $\bar{b}$, then
$\bar{a}\cdot\bar{b}$ is the reduction of the product $ab\in A$. 
\end{prop}

\begin{proof}
Let $a,b\in A$ whose images $\bar{a},\bar{b}$ in $A/A\cdot J$ are
$\mathfrak{m}$-invariant. As $A/A\cdot J$ is $h$-torsion free,
the $\mathfrak{m}$-invariants are elements invariant under $\text{ad}(\rho(m))$
for all $m\in\mathfrak{m}$. Thus we see that, for any $m\in\mathfrak{m}_{l}$,
$\rho(m)a-a\rho(m)=j_{m}$ and $\rho(m)b-b\rho(m)=k_{m}$ for some
$j_{m},k_{m}\in A\cdot J$. Choose a $\mathbb{C}$-basis for $\mathfrak{m}$,
which we call $\{x_{i}\}$. Then any element of $U_{h}(\mathfrak{m})$
can be represented as a formal series 
\[
\sum_{I,i}c_{I,i}h^{i}x^{I}
\]
where $c_{I,i}\in\mathbb{C}$, $x^{I}=x_{!}^{i_{1}}\dots x_{n}^{i_{n}}$
for a multi-index $I$, and $i\to\infty$ as $|I|\to\infty$. Repeatedly
applying the above formulas for the action of $\rho(x)$, we see that
$\rho(x^{I})b-b\rho(x^{I})=\alpha_{x^{I}}+k_{x^{I}}$ where $k_{x^{I}}\in A\cdot J$
and $\alpha_{x^{I}}$ is an element of $A\cdot J$ (this is because
$J$ is a two-sided ideal of $U_{h}(\mathfrak{m})(0)$, so that $A\cdot J\cdot\rho(m)\in A\cdot J$
for all $m\in\mathfrak{m}$). Therefore, as $\rho(h)=h$ we have
\[
\rho(\sum_{I,i}c_{I,i}h^{i}x^{I})b=\sum_{I,i}c_{I,i}h^{i}(b\rho(x^{I})+\alpha_{x^{I}})=b\sum_{I,i}c_{I,i}h^{i}+\sum_{I,i}c_{I,i}h^{i}\alpha_{x^{I}}
\]
and as the latter sum is convergent in the complete ideal $A\cdot J$,
we see that 
\[
\rho(u)b-b\rho(u)\in A\cdot J
\]
for any $u\in U_{h}(\mathfrak{m})$. 

Now, consider any representatives for the classes $\bar{a},\bar{b}$;
call them $a+j_{1}$, $b+k_{1}$. Then 
\[
(a+j_{1})(b+k_{1})=ab+j_{1}b+ak_{1}+j_{1}k_{1}
\]
and, as $j_{1}\in A\cdot J$, we obtain $j_{1}b\in A\cdot J$ by commuting
elements of form $\rho(u)$ past the $b$ via the above claim. Therefore
$(a+j_{1})(b+k_{1})$ is congruent to $ab$ in $A/A\cdot J$ and the
result follows. 
\end{proof}
Now, in the main case of interest in this paper, $R=A/h$ will be
the ring of coordinates of an algebraic variety, on which $M$ acts
algebraically. In that case, as $M$ is connected, one has that the
$M$-invariants in $R/(\overline{J})$ are the same as the $\mathfrak{m}$-invariants.
As $A/A\cdot J$ is an infinitesimal deformation of $R/(\overline{J})$,
a quick proof by induction shows that $(A/A\cdot J)^{M}=(A/A\cdot J)^{\mathfrak{m}}$.
Thus we obtain 
\begin{cor}
\label{cor:Ham-reduction-group}With hypotheses as above, the complete
$\mathbb{C}[[h]]$-module $(A/A\cdot J)^{M}$ is naturally an algebra
so that, if $a,b\in A$ are elements whose images in $A/A\cdot J$
are denoted $\bar{a}$ and $\bar{b}$, then $\bar{a}\cdot\bar{b}$
is the reduction of the product $ab\in A$. 
\end{cor}

This type of Hamiltonian reduction will occur throughout the paper. 

\section{Differential Operators and Quantization}

Let $X$ be a smooth complex algebraic variety. Then the sheaf of
differential operators on $X$, $D_{X}$, is a sheaf of filtered algebras
whose associated graded sheaf is isomorphic to $\pi_{*}(O_{T^{*}X})$
where $T^{*}X$ is the cotangent bundle to $X$, and $\pi:T^{*}X\to X$
is the natural map (see \cite{key-1} for details). So $D_{X}$ is
a quantization of the cotangent bundle of $X,$ but it is only local
on $X$, not $T^{*}X$. To correct this, we introduce the following
\begin{defn}
\label{def:D_h}(c.f \cite{key-6}) Let $X$ be an affine complex
algebraic variety. The we define the algebra of asymptotic differential
operators on $X$, $D_{h}(X)(0),$ to be the $h$ completion of the
algebra generated by $O_{X}$, the global vector fields $\Theta_{X}$,
and the variable $h$, subject to the relations $f_{1}*f_{2}=f_{1}f_{2}$
($f_{i}\in O_{X}$), $f*\xi=f\xi$, $\xi*f-f*\xi=h\xi(f)$, $\xi_{1}\xi_{2}-\xi_{2}\xi_{1}=h[\xi_{1},\xi_{2}]$,
($f\in O_{X},\xi_{i}\in\Theta_{X}$). The algebra $D_{h}(X)(0)$ is
a quantization of $T^{*}X$ in the sense of the definition given below.
We may apply localization to this algebra to obtain a sheaf on $T^{*}X$;
for a general algebraic variety we glue this construction to obtain
the sheaf of asymptotic differential operators $D_{h}(X)(0)$. We
again emphasize that this is a sheaf on $T^{*}X$, not on $X$. 
\end{defn}

The general context for this definition is given by 
\begin{defn}
\label{def:Quantization}(c.f \cite{key-5}) Let $Y$ be a Poisson
variety, i.e., an complex algebraic variety equipped with a Poisson
bracket on the structure sheaf. A quantization of $Y$, $O_{h}$,
is a sheaf of associative, flat $\mathbb{C}[[h]]$ algebras on $Y$
that is complete with respect to the $h$-adic topology and equipped
with an isomorphism $O_{h}/hO_{h}\tilde{\to}O_{Y}$. This gives $O_{Y}$
the structure of a sheaf of Poisson algebras, and we demand that this
structure agrees with the given one. 
\end{defn}

Most of the time (though not always!) in this paper, it will be the
case that $Y$ is smooth and the Poisson structure in question comes
from a symplectic form on $Y$. In the case of \ref{def:D_h},
the symplectic variety in question is of the form $T^{*}X$. 

As the algebras appearing here are $h$-torsion free, it is sometimes
convenient to invert $h$. We define $D_{h}(X):=D_{h}(X)(0)[h^{-1}]$
for any algebraic variety; this is a $\mathbb{C}((h))$-linear sheaf
on $T^{*}X$. Although not a quantization, this is the sheaf of algebras
that we will actually use in this paper, for reasons that will become
clear in the next section. 

We note at this point that this sheaf is considered (in a somewhat
different notation) in the paper \cite{key-7}. There, they introduce
the formalism of $W$-algebras (no relation to the $W$-algebras in
section 1!). To avoid confusion, we will call them $QDO$-algebras,
standing for quantized differential operator algebras. We recall now
the 
\begin{defn}
\cite{key-7} Let $X$ be a smooth holomorphic symplectic variety
of dimension $2n$. A $QDO$-algebra on $X$ is a sheaf of $\mathbb{C}((h))$-linear
algebras, $D_{h}$, such that for each $x\in X$, there exists an
open neighborhood $U$ of $x$ and a symplectic holomorphic morphism
$\phi:U\to T^{*}\mathbb{C}^{n}$ such that $D_{h}|_{U}\tilde{\to}\phi^{*}D_{h}(\mathbb{C}^{n})$. 
\end{defn}

Although this is a convenient definition in the analytic topology,
in the Zariski topology it is very poorly behaved\footnote{We would like to thank the referee for emphasizing this point.}.
Therefore we will not use this definition in this paper; however,
we should note that several of the basic techniques of \cite{key-7}
do carry over to this situation; in particular, many of their arguments
only rely on the fact that a $QDO$ is a quantization in the sense
of \ref{def:Quantization}; when we quote results from this
paper (as we do a few times in the sequel), we only quote results
of this type.

Next, we define the categories of modules over quantizations that
we will consider. Let us recall from \cite{key-8}, theorem 1.2.5,
that if $O_{h}$ is a quantization of a smooth variety $Y$, then
$O_{h}$ is a locally noetherian, stalk-wise noetherian sheaf of algebras. 
\begin{defn}
Let $O_{h}$ be a quantization of the smooth Poisson algebraic variety
$Y$. An $O_{h}$-module $M(0)$ is coherent if it is locally finitely
generated. Further, any such module is automatically complete in the
$h$-adic topology (by \cite{key-8}, theorem 1.2.5, part 3). 

An $O_{h}[h^{-1}]$ module $M$ is said to be coherent if there exists,
globally, a coherent $O_{h}[h^{-1}]$ module, $M(0)$, such that $M\tilde{=}M(0)[h^{-1}]$.
An $O_{h}[h^{-1}]$-module is said to be quasi-coherent if it is a
direct limit of coherent modules; i.e., we simply define $Mod^{qc}(O_{h}[h^{-1}])$
as the ind-category of $Mod^{coh}(O_{h}[h^{-1}])$. 
\end{defn}

As noted above, we may refer the reader to \cite{key-8} (chapter
1) for details about modules over quantized algebras in a very general
context. In particular, we note that our definition of ``coherent''
for $D_{h}(0)$-modules agrees with the one given there. Applying
again their theorem 1.2.5, we have
\begin{lem}
An $O_{h}$-module $M(0)$ is coherent iff $M(0)$ is $h$-complete
and $h^{n}M(0)/h^{n+1}M(0)$ is a coherent $O_{Y}$ module for all
$n\geq0$. This is equivalent to the condition of $M(0)$ having bounded
$h$-torsion and $M(0)/hM(0)$ being coherent over $O_{Y}$. 
\end{lem}

We also note that the categories $Mod^{coh}(O_{h})$, $Mod^{coh}(O_{h}[h^{-1}])$
and $Mod^{qc}(O_{h})$ are abelian; the first two by the noetherian
property of $O_{h}$, and the last because it is an ind-category. 

To finish this section, we note a key fact about the cohomology of
modules over the algebra $O_{h}$. This is lemma $2.12$ in \cite{key-7},
and its proof goes over mutatis mutandis to the algebraic situation;
one may also consult (the proof of) \cite{key-8}, theorem 1.2.5,
part 5. 
\begin{lem}
\label{lem:basic-vanishing} Let $M(0)$ be a coherent $O_{h}$ module.
Assume that $H^{i}(X,M(0)/hM(0))=0$ for $i>0$. Then we have that 

i) The natural map $\Gamma(X,M(0))\to\Gamma(X,M(0)/hM(0))$ is surjective.

ii) $H^{i}(X,M(0))=0$ for $i>0$. 
\end{lem}

\section{Equivariance }

We suppose now that we have an algebraic group $G$ acting algebraically
on our Poisson variety $Y$, in such a way that the group action respects
the Poisson bracket. We wish to define equivariant versions of everything
introduced in the previous section. We start with a general
\begin{defn}
\label{def:Equivariance-for-O}Let $O_{h}$ be a quantization of $Y$.
Then $O_{h}$ is said to be $G$-equivariant if each sheaf $O_{h}/h^{n}O_{h}$
(for $n\geq0$) admits a $G$-equivariant structure (as a sheaf of
algebras; i.e., we demand that the action of $G$ respect the multiplication),
in such a way that the natural maps $O_{h}/h^{n+1}O_{h}\to O_{h}/h^{n}O_{h}$
are $G$-morphisms. We demand that $h$ be stable under the action
of $G$ in the sense that, on global sections, $\rho_{g}^{-1}(h)=\chi(g)h$,
where $\chi$ is an algebraic character of $G$ (which will usually
be the trivial character).
\end{defn}

In particular, this definition gives us isomorphisms $O_{h}\tilde{\to}\rho_{g}^{-1}O_{h}$
(where $\rho_{g}:Y\tilde{\to}Y$ is the map associated to $g\in G$)
for all $g\in G$. This definition extends immediately to the algebra
$O_{h}[h^{-1}]$; we simply extend the action by demanding that $G$
act on $h^{-1}$ by the inverse of the character $\chi$.

To obtain equivariance conditions for coherent modules, we let $M\in Mod^{coh}(O_{h})$
and let $M(0)$ be a lattice. We further suppose that $O_{h}$ is
a $G$-equivariant sheaf in the above sense. Then we have
\begin{defn}
$M(0)$ is a quasi-$G$-equivariant $O_{h}$-module if each sheaf
$M(0)/h^{n}M(0)$ (for $n\geq0$) is $G$-equivariant as a module
over $O_{h}/h^{n}$; i.e., there is a $G$-action on $M(0)/h^{n}M(0)$
that is compatible with the $G$-action on $O_{h}/h^{n}$, in the
sense that the action map 
\[
O_{h}/h^{n}\otimes M(0)/h^{n}M(0)\to M(0)/h^{n}M(0)
\]
is $G$-equivariant for each $n$; and this action makes $M(0)/hM(0)$
into an equivariant coherent sheaf. We demand that the natural quotient
maps are $G$-morphisms. We demand that $h$ be stable under the action
of $G$ in the sense that, on global sections, $\rho_{g}^{-1}(h)=\chi(g)h$
where $\chi$ is an algebraic character of $G$ (which will usually
be the trivial character). 

We say that $M$ is $G$-equivariant if it admits a $G$-action, and
there is a $G$-stable lattice $M(0)$ which is $G$-equivariant in
the above sense.
\end{defn}

If $M$ is a quasicoherent $O_{h}$-module equipped with an action
of $G$, we say that it is $G$-equivariant if it is a direct limit
of $G$-equivariant coherent modules. 

This definition gives us categories $Mod^{G,?}(O_{h})$ (where $?$
is coherent or quasicoherent), where we demand that the morphisms
respect the $G$-structure. %

If we go back to our primary example where our symplectic variety
is $T^{*}X$, then the sheaf $D_{h}(X)$ is $\mathbb{C}^{*}$- equivariant
for the $\mathbb{C}^{*}$action on $T^{*}X$ given by dilation on
the fibers of $\pi:T^{*}X\to X$. Therefore, we can consider the sheaf
of fixed points, denoted $D_{h}(X)^{\mathbb{C}^{*}}$. This is a sheaf
of algebras on $T^{*}X$, and one verifies $D_{h}(X)^{\mathbb{C}^{*}}\tilde{=}E_{X}$,
where $E_{X}$ denotes the sheaf of formal differential operators
on $T^{*}X$. Then the functor $M\to\pi_{*}(M)^{\mathbb{C}^{*}}$
provides an equivalence of categories between the category of $\mathbb{C}^{*}$-equivariant
coherent $D_{h}(X)$ modules, and that of coherent $D_{X}$ modules
(this fact, or, rather, its analytic analogue, is discussed in \cite{key-7},
section 2.3.3). 

In fact, $\mathbb{C}^{*}$-equivariant modules will play a major role
in this paper. At this point, we note a few facts: we will consider
only $\mathbb{C}^{*}$ actions that act on $h$ as some $t^{n}$ for
$n\geq0$, and we can assume that $n=1$ (if not, we can simply replace
the ground field $\mathbb{C}((h))$ by $\mathbb{C}((h^{1/n}))$, base
change everything to this field, and demand that $\mathbb{C}^{*}$
acts on $h^{1/n}$ as $t$). 
\begin{lem}
\label{lem:C-star-actions} Given such an action on a $D_{h}$-module
$M$, for any $U\subseteq X$ which is affine, open and $\mathbb{C}^{*}$-invariant,
$\Gamma(U,M)\neq0$ implies $\Gamma(U,M)^{\mathbb{C}^{*}}\neq0$.
Further, if $M$ is coherent, $\Gamma(U,M)^{\mathbb{C}^{*}}$ is generated
by finitely many sections as a $\Gamma(U,D)^{\mathbb{C}^{*}}$ module. 
\end{lem}

\begin{proof}
These facts are proved by using the definition of equivariance given
above. To show the existence of an invariant section, we assume WLOG
that $M$ is coherent, using that, over an affine open set, a quasicoherent
$\mathbb{C}^{*}$-equivariant module is a limit of its equivariant
coherent submodules. So, each surjection $\Gamma(U,M(0))/h^{n+1}\Gamma(U,M(0))\to\Gamma(U,M(0))/h^{n}\Gamma(U,M(0))$
admits a $\mathbb{C}^{*}$-invariant splitting. Therefore, we can
choose $\mathbb{C}^{*}$-homogeneous sections in $\Gamma(U,M(0))$
whose images generate $\Gamma(U,M(0))/h\Gamma(U,M(0))$. So $\Gamma(U,M)\neq0$
implies that $h^{n}s\neq0$ for all $n\in\mathbb{Z}$, for at least
one of these sections $s$; and then choosing the correct $n$ gives
a $\mathbb{C}^{*}$-invariant section. The fact about finitely generated
modules follows from the Nakayama lemma and the fact that $\Gamma(U,M(0))/h\Gamma(U,M(0))$
is finitely generated over $\Gamma(U,D_{h}(0))/h\Gamma(U,D_{h}(0))$
for coherent modules by writing out the action of $D_{h}$ on $M$. 
\end{proof}

\section{\label{sec:Hamiltonian-reduction}Hamiltonian Reduction for Quantizations}

Let $H$ be a connected affine algebraic group, with Lie algebra $\mathfrak{h}$,
and suppose that $H$ acts on the algebraic variety $X$. Then the
induced action of $H$ on $T^{*}X$ is Hamiltonian (see \cite{key-2}
page 44 for details) and so there exists an $H$-equivariant moment
map $\mu:T^{*}X\to\mathfrak{h}^{*}$. In fact, we can describe explicitly
the comorphism on functions as follows: any $y\in\mathfrak{h}$ gives
rise to an algebraic vector field on $X$, denoted $\xi_{y}$, which
in turn gives rise to a regular function on $T^{*}X$, called $f_{y}$,
via the natural pairing of tangent and cotangent vectors. This map
extends uniquely to an algebra morphism $O(\mathfrak{h}^{*})=S\mathfrak{h}\to O(T^{*}X).$ 

Let $\chi\in\mathfrak{h}^{*}$ be a character (i.e., suppose that
$\chi([\mathfrak{h},\mathfrak{h}])=0$). Then $\mu^{-1}(\chi)$ is
an $H$ invariant closed subvariety of $T^{*}X$. Suppose that $\mu$
has surjective differential at all points in $\mu^{-1}(\chi)$, so
that $\mu^{-1}(\chi)$ is a smooth subvariety. Suppose further that
there exists a smooth quotient $\mu^{-1}(\chi)/H$ in the sense that
there exists a morphism $p:\mu^{-1}(\chi)\to\mu^{-1}(\chi)/H$ making
$\mu^{-1}(\chi)$ a principal $H$-bundle (in the Zariski topology)
over $\mu^{-1}(\chi)/H$ , and we assume that this quotient admits
a symplectic form compatible with the reduction. Then this quotient
variety is called the Hamiltonian reduction of $T^{*}X$ with respect
to $\chi$. 

In this section we're going to explain how the Hamiltonian reduction
procedure for quantizations will allow us to obtain quantizations
of various spaces $\mu^{-1}(\chi)/H$. In fact, the procedure we will
discuss sometimes works for quantizations of Poisson varieties as
well; so we move now to that level of generality; however, the case
discussed above is the key one to keep in mind.

Suppose now that $Y$ is a smooth Poisson variety, with quantization
$O_{h}$. Suppose that $Y$ admits an action of $H$, and that $O_{h}$
is $H$-equivariant in the sense of \ref{def:Equivariance-for-O}.
Suppose further that $O_{h}$ admits a quantum comoment map; i.e.,
there is map of algebras\footnote{To make sense of this we regard $U_{h}(\mathfrak{h})(0)$ as a constant
sheaf on $Y$.} $\rho:U_{h}(\mathfrak{h})(0)\to O_{h}$ satisfying the condition
of \ref{def:Ham-Reduc-for-quantized}. In particular there is
an $H$-equivariant map $\mu:Y\to\mathfrak{h}$ (which arises by taking
$h\to0$ in the quantum comoment map); in the situation discussed
above of an action of $H$ on $X$ and $O_{h}=D_{h}$ (on $T^{*}X$)
then these criteria are fulfilled. The quantum comoment map takes
$x\in\mathfrak{h}$ to the vector field $\xi_{x}$, considered as
a section of $D_{h}$. 

Let $\chi:U_{h}(\mathfrak{h})(0)\to\mathbb{C}[[h]]$ be a continuous
character. Such a map is determined by its restriction to $\mathfrak{h}$,
so specifying $\chi$ is the same as specifying a $\mathbb{C}$-linear
map, which we will abusively denote $\chi:\mathfrak{h}/[\mathfrak{h},\mathfrak{h}]\to\mathbb{C}[[h]]$.
Let $\chi_{0}$ denote the reduction mod $h$ of $\chi$; note that
even if $\chi$ is nontrivial then $\chi_{0}=0$ is possible. The
sheaf of $O_{h}$-modules $L_{\chi}:=O_{h}/\sum_{x\in\mathfrak{h}}O_{h}\cdot(\xi_{x}-\chi(x))$
is then supported on $\mu^{-1}(\chi_{0})$. Suppose, as above, that
$\mu^{-1}(\chi_{0})/H$ exists, and let $p:\mu^{-1}(\chi_{0})\to\mu^{-1}(\chi_{0})/H$
denote the quotient map; which we assume to be an $H$-torsor. Then
we make the 
\begin{defn}
The sheaf $p_{*}(L_{\chi})^{H}$ (which, under the assumption that
$L_{\chi}$ is $h$-torsion free, is a sheaf of algebras on $\mu^{-1}(\chi_{0})/H$
by \ref{cor:Ham-reduction-group}) is called the quantum Hamiltonian
reduction of $O_{h}$, with respect to $\chi$. 
\end{defn}

At this level of generality, it seems that there is not much we can
say about about $p_{*}(L_{\chi})^{H}$. For instance, it is not obvious\footnote{We thank the referee for pointing this out.}
whether or not it is actually a quantization of $\mu^{-1}(\chi_{0})/H$.
However, in the two main cases of interest in this paper, we shall
see that this is indeed the case. 

Now let us turn to the relationship between equivariant modules and
Hamiltonian reduction. 

Let $M(0)$ be an $H$-equivariant coherent $O_{h}$-module, in the
sense of the previous section. Let $\beta:\mathfrak{h}\to\text{End}_{\mathbb{C}[[h]]}(M(0))$
denote the derivative of the action map of $H$ on $O_{h}$. We note
that, in the case that $Y=T^{*}X$ (with $H$ acting as above) and
$O_{h}=D_{h}(X)=M(0)$, then $h\cdot\beta(x)$ is the map $[\xi_{x},\cdot]$,
where $\xi_{x}$ is the vector field associated to $x$ via the action,
considered as a section of $D_{h}$. 
\begin{defn}
\label{def:twist}Let $\chi:\mathfrak{h}/[\mathfrak{h},\mathfrak{h}]\to\mathbb{C}[[h]]$
be a $\mathbb{C}$-linear map. We say that $M(0)$ has twist $\chi$
if the map $h\beta(x)-\rho(x)$ is equal to $\chi(x)$ on $M(0)$,
for all $x\in\mathfrak{h}$. 
\end{defn}

We denote by $Mod_{\chi}^{coh,H}(O_{h})$ (resp. $Mod_{\chi}^{qcoh}(O_{h})$)
the category of coherent (resp. quasicoherent) modules over $O_{h}$
with twist $\chi$; this is a full abelian subcategory of the category
of all equivariant $O_{h}$-modules. Observe that $L_{\chi}$ has
twist $\chi$, by the definition of the comoment map. Furthermore,
unpacking the definition yields $\mathcal{H}om(L_{\chi},M(0))=M(0)^{\mathfrak{h}}$
for any $h$-torsion-free $M(0)$ with twist $\chi$, and since $H$
is connected this means $\mathcal{H}om(L_{\chi},M(0))=M(0)^{H}$ as
well. 

Then we have the following 
\begin{prop}
\label{prop:Reduction-equivalence}In the situation above, suppose
that $L_{\chi}$is $h$-torsion free. Suppose further that $p_{*}(L_{\chi})^{H}$
is a quantization of the smooth variety $Z=\mu^{-1}(\chi_{0})/H$,
and suppose we have
\[
\mathcal{E}xt_{O_{h}}^{i}(L_{\chi}[h^{-1}],L_{\chi}[h^{-1}])=0
\]
for $i>0$. Then

1) We have equivalences of categories $Mod^{?}(p_{*}(L_{\chi})^{H}[h^{-1}])\tilde{\to}Mod_{\chi}^{?,H}(O_{h}[h^{-1}])$
(where $?$ stands for coherent or quasicoherent) given by 
\[
M\to\mathbb{H}^{\perp}(M):=L_{\chi}[h^{-1}]\otimes_{p^{-1}(p_{*}(L_{\chi})^{H})[h^{-1}]}p^{-1}M
\]
 for $M$ in $Mod^{?}(p_{*}(L_{\chi})^{H})$, and $S\to\mathbb{H}(S)=p_{*}(\mathcal{H}om(L_{\chi}[h^{-1}],S))^{H}$
for $S$ in $Mod_{\chi}^{?,H}(O_{h})$ (where in the second functor
we use $\mathcal{H}om$ to mean sheaf hom).

2) Suppose that there exists a $\mathbb{C}^{*}$-action on $X$ that
preserves $\mu^{-1}(\chi_{0})$ and that this drops to an action on
$Z$ in such a way that $p_{*}(L_{\chi})^{H}$ is $\mathbb{C}^{*}$-equivariant.
Suppose that $\mathbb{C}^{*}$ acts on $H$ in such a way that $H\rtimes\mathbb{C}^{*}$
acts on $X$. Then we have an equivalence 
\[
Mod^{?,\mathbb{C}^{*}}(p_{*}(L_{\chi})^{H})\tilde{\to}Mod_{\lambda}^{?,\mathbb{C}^{*}\times H}(O_{h})
\]
which is given by the same formulae. 
\end{prop}

We will prove this momentarily; our argument uses the theory of quantizations
of algebraic varieties as developed in \cite{key-25}, and in particular
ideas of the proof of the Hamiltonian reduction theorem there; though,
due the group action, this case is technically different and is not
implied by the theorem there (or vice versa). Before giving the details,
we would like to comment on the relation with the analogous statement
proved by Kashiwara and Rouquier, namely \cite{key-7}, proposition
2.8. Aside from the technical difference of working in the analytic
topology as opposed to the Zariski (which is ultimately a minor issue),
the main difference between their construction and ours is that we
consider a more general notion of twist then they do. To make the
two situations consistent let us suppose that we are working over
$T^{*}X$ with the quantization $D_{h}$. Let $\lambda\in\mathfrak{h}^{*}$
be a character. Then their definition of $M=M(0)[h^{-1}]$ having
twist $\lambda$ is that $\beta(x)-h^{-1}\xi_{x}$ acts as $\lambda(x)$,
for all $x\in\mathfrak{h}$. Multiplying through by $h$, we see that
this is equivalent, in our notation, to $M(0)$ having twist $\chi=h\lambda$.
In particular, their version of Hamiltonian reduction always ends
up as a quantization of $\mu^{-1}(0)/H$; no other fibres of $\mu$
are possible. As our interest is in Slodowy slices, this would be
too restrictive for our purposes. 

Furthermore, as we shall see below, in the case of $D_{h}$ over $T^{*}X$,
the various conditions of the lemma are automatically satisfied. For
another case in which the conditions hold, see \ref{prop:Ham-reduction-on-g}
below.

Now let us proceed to the 
\begin{proof}
(of \ref{prop:Reduction-equivalence}) We will prove the result
for coherent modules; the quasicoherent case follows formally. In
addition, part $2$ follows by the same method as part $1$, so we
will concentrate on the first case. Note that the functors in question
form an adjoint pair by the usual hom-tensor adjunction. 

Let us denote $p_{*}(L_{\chi})^{H}:=O_{Z,h}$, as it is by assumption
a quantization of $Z$. 

Recall that $L_{\chi}/hL_{\chi}=O_{\mu^{-1}(\chi_{0})}$. Since $R^{i}p_{*}(O_{\mu^{-1}(\chi_{0})})=0$,
we have that $R^{i}p_{*}(L_{\chi})=0$ for $i>0$ as well (by the
argument of lemma 2.12 in \cite{key-7}). Therefore 
\[
p_{*}(L_{\chi})/hp_{*}(L_{\chi})\tilde{=}p_{*}(L_{\chi}/hL_{\chi})
\]
is locally projective as a module over $O_{Z}$ (since $\mu^{-1}(\chi_{0})$
is an $H$-torsor over $Z$). From \cite{key-8} corollary 1.6.7,
we see that $p_{*}(L_{\chi})$ is faithfully flat over $O_{Z,h}$.
Therefore the functor $\mathbb{H}^{\perp}$ is exact and conservative.
It follows also that the map 
\[
S\to\mathbb{H}\circ\mathbb{H}^{\perp}(S)
\]
is an isomorphism for all coherent $S$ over $O_{Z,h}$; indeed, since
$Z$ is smooth we may take locally a finite projective resolution
of $S$; so the result follows from the claim for $O_{Z,h}$ itself;
but this holds by definition. 

Next we claim that $\mathcal{E}xt_{O_{h}}^{i}(L_{\chi}[h^{-1}],\mathbb{H}^{\perp}(S))=0$
for $i>0$ and for any coherent $O_{h,Z}[h^{-1}]$-module $S$. For
this, we use again that $S$ is locally of finite homological dimension.
Therefore we may proceed by induction on the homological dimension;
the base case being exactly the assumption (compare, e.g., the argument
of \cite{key-25}, claim 4.23). 

Now we consider the functor $\mathbb{H}$. We claim that it conservative
on $Mod_{\chi}^{coh,H}(O_{h}[h^{-1}])$. It suffices to prove this
for a lattice $M(0)$. Note that the condition of having twist $\chi$
means that $\mathcal{H}om(L_{\chi},M(0))=M(0)^{H}$. However, $M(0)/h$
is an $H$-equivariant coherent sheaf which is set-theoretically supported
on the torsor $\mu^{-1}(\chi_{0})$, so $M(0)/h\neq0$ implies that
$M(0)$ has a nontrivial space of $H$-invariants, and that, in fact,
$M(0)/hM(0)$ is the pullback of a sheaf on $\tilde{Z}$, an infinitesimal
thickening of $Z$. 

Proceeding by induction, consider the sequence 
\[
0\to M(0)/h\to M(0)/h^{n}\to M(0)/h^{n-1}\to0
\]
As there are no higher derived $H$-invariants of $M(0)/h$ by the
above, one concludes that $M(0)/h^{n}$ has a nonzero space of $H$-invariants,
which surjects onto the $H$-invariants of $M(0)/h^{n-1}$. Taking
the inverse limit, we obtain the claim. 

We now finish the proof by showing that the map 
\[
\mathbb{H}^{\perp}\circ\mathbb{H}(M)\to M
\]
is an isomorphism. Write $K$ and $C$ for the kernel and cokernel
of this map, respectively. By the adjunction between $\mathbb{H}$
and $\mathbb{H}^{\perp}$, we have that 
\[
\mathbb{H}\circ\mathbb{H}^{\perp}\circ\mathbb{H}(M)\to\mathbb{H}(M)
\]
is an isomorphism. Since $\mathbb{H}$ is left exact (it is a hom
functor), this forces $\mathbb{H}(K)=0$, so $K=0$ as $\mathbb{H}$
is conservative. Thus we have that 
\[
0\to\mathbb{H}^{\perp}\circ\mathbb{H}(M)\to M\to C\to0
\]
is a short exact sequence; note that this implies $\mathbb{H}^{\perp}\circ\mathbb{H}(M)$
is coherent over $O_{h}[h^{-1}]$; from this it follows that $\mathbb{H}(M)$
is coherent over $O_{Z,h}$. Now, apply $\mathbb{H}$ to the exact
sequence above, and use the fact that $\mathcal{E}xt_{O_{h}}^{1}(L_{\chi}[h^{-1}],(\mathbb{H}^{\perp}\circ\mathbb{H}(M)))=0$
as proved above, to see that $\mathbb{H}(C)=0$, so that $C=0$ and
the result follows. 
\end{proof}
Now we need to explain that the hypotheses actually hold when we want
them to: 
\begin{lem}
\label{lem:Good-Conditions-For-T^*}Suppose that, in the situation
of \ref{prop:Reduction-equivalence}, we have that $Y=T^{*}X$,
$O_{h}=D_{h}$, and the Hamiltonian group action of $H$ comes from
an action of $H$ on $X$. Suppose also that $\mu^{-1}(\chi_{0})/H$
is a smooth symplectic variety. Then the hypotheses of \ref{prop:Reduction-equivalence}
are satisfied. 
\end{lem}

\begin{proof}
Essentially, this is a consequence of the formal Darboux lemma for
the quantization $D_{h}$. Let $x\in\mu^{-1}(\chi_{0})$, and denote
by $\widehat{D}_{h,x}$ the formal completion of $D_{h}$ at $x$,
and $\widehat{O}_{x}$ the formal completion of the structure sheaf
of $T^{*}X$ at $x$; the formal neighborhood of $x$ in $T^{*}X$
is then $\text{Spf}(\widehat{O}_{x})$. Applying the local description
of the Hamiltonian reduction in the smooth case (as recalled in \cite{key-7},
lemma 2.7), we have an isomorphism of formal symplectic schemes 
\[
\text{Spf}(\widehat{O}_{x})\tilde{=}\text{Spf}(\widehat{O}_{T^{*}H,e}\widehat{\otimes}\widehat{O}_{Z,p(x)})
\]
where by $\widehat{O}_{T^{*}H,e}$ we mean the completion of $O_{T^{*}H}$
at the identity element, and by $\widehat{O}_{Z,p(x)}$ we mean the
completion of $Z$ at $p(x)$; this isomorphism is compatible with
the $H$-action.

Therefore, by the basic structure theory for $\widehat{D}_{h}$ (c.f.
\cite{key-5}, lemma 1.5, or, somewhat more elaborately, \cite{key-25},
proposition 4.9), we have a corresponding isomorphism of quantizations
\[
\widehat{D}_{h}\tilde{=}\widehat{D}_{h,H,e}\widehat{\otimes}\widehat{W}_{h,Z}
\]
where $\widehat{D}_{h,H,e}$ denotes the formal completion of $h$-differential
operators on $H$ at the identity of $T^{*}H$, and $\widehat{W}_{h,Z}$
is a formal quantization of $\widehat{O}_{Z}$; it is necessarily
isomorphic to a formal Weyl algebra by \cite{key-5}, lemma 1.5. By
construction we have 
\[
\widehat{D}_{h,H,e}\tilde{=}\widehat{W}_{h}(\mathfrak{h}\oplus\mathfrak{h}^{*})
\]
where on the right hand side we have the $h$-completed formal Weyl
algebra on the symplectic vector space $\mathfrak{h}\oplus\mathfrak{h}^{*}$.
Thus in total we have 
\[
\widehat{D}_{h}\tilde{=}\widehat{W}_{h}(\mathfrak{h}\oplus\mathfrak{h}^{*})\widehat{\otimes}\widehat{W}_{h,Z}
\]

Under this isomorphism the formal completion of the left ideal $\sum_{x\in\mathfrak{h}}D_{h}\cdot(\xi_{x}-\chi(x))$
is simply given by $\widehat{D}_{h}\cdot\mathfrak{h}$. Therefore,
if we formally complete $L_{\chi}$ at $x$ we have 
\[
\widehat{L}_{\chi}=\widehat{D}_{h}/\widehat{D}_{h}\cdot\mathfrak{h}
\]
which is evidently $h$-torsion free. This shows that $L_{\chi}$
is $h$-flat. Furthermore, the invariants 
\[
\widehat{L}_{\chi}^{H}\tilde{=}\widehat{W}_{h,Z}
\]
satisfy$\widehat{L}_{\chi}^{H}/h\tilde{=}\widehat{O}_{Z}$; from this
it follows that the natural map $p_{*}(L_{\chi})^{H}/h\to O_{Z}$
is an isomorphism, as it is so after each formal completion; i.e.,
$p_{*}(L_{\chi})^{H}$ is a quantization of $O_{Z}$. 

Finally, to obtain the statement on Ext vanishing, we note that it
too can be checked upon formal completion (compare, e.g. \cite{key-25},
theorem 4.17). Using the isomorphism $\widehat{D}_{h}\tilde{=}\widehat{W}_{h}(\mathfrak{h}\oplus\mathfrak{h}^{*})\widehat{\otimes}\widehat{W}_{h,Z}$
we see that the needed vanishing reduced to the triviality of the
de Rham cohomology of the affine space $\mathfrak{h}^{*}$ (use \cite{key-25},
lemma 4.14), whence the result.
\end{proof}
\begin{rem}
As this proof is fundamentally local, this result carries over without
change to the case where $O_{h}$ is a quantization of $T^{*}X$,
which is locally (but not necessarily globally) isomorphic to $D_{h}$.
We will encounter such quantizations in the next two sections. 
\end{rem}

Finally we finish off this section with
\begin{example}
Suppose $H$ is a connected, affine algebraic group, $B\leq H$ a
connected algebraic subgroup, with $Lie(B)=\mathfrak{b}^{*}$. Then
we have the natural left and right actions of $B$ on $H$; which
extend to actions on $T^{*}H$. The moment map (for the right action)
in this case can be described as follows: we have an isomorphism $T^{*}H\tilde{=}H\times\mathfrak{h}^{*}$,
and thus a map to $\mathfrak{h}^{*}$ via $(h,\xi)\to ad^{*}(h)(\xi)$
(this is the moment map for the action of $H$ on $T^{*}H$, denoted
$\mu^{'}$). So the moment map $\mu$ for \textbf{$B$ }is given by
the composition $T^{*}H\to\mathfrak{h}^{*}\to\mathfrak{b}^{*}$. So,
we can describe $\mu^{-1}(\lambda)$ by first noting that $res^{-1}(\lambda)=\lambda+\mathfrak{b}^{\perp}$
(where $\mathfrak{b}^{\perp}$ denotes the annihilator of $\mathfrak{b}$
in $\mathfrak{h}^{*}$). The inverse image of this space under $\mu^{'}$
is the closed subvariety $\{(h,ad^{*}(h^{-1})(\lambda+\mathfrak{b}^{\perp}))\in H\times\mathfrak{h}^{*}\}$.
Now, if $\lambda=0$, then it is immediate that the quotient of this
variety by $B$ is isomorphic to $H\times_{B}\mathfrak{b}^{\perp}\tilde{=}T^{*}(H/B).$
For general $\lambda$, we obtain a algebraic variety $T^{*}(H/B)^{\lambda}$
called a twisted cotangent bundle. We note that if $H=G$, and $B$
is a Borel subgroup, then $G/B$ is the full flag variety. In this
case, we can consider characters of $\mathfrak{b}$ which come from
$\mathfrak{h}$ via extension by zero. If such a $\lambda$ is an
integral character, then the variety $T^{*}(G/B)^{\lambda}$ is isomorphic
to $T^{*}(G/B)$. 
\end{example}

\section{Localization}

We now apply the formalism of the above sections to the finite $W$
algebras. To do this, we'll first discuss (a version of) the classical
Beilinson-Bernstein localization theorem. We start with the cotangent
bundle $T^{*}G$, and the sheaf of asymptotic differential operators
$D_{h}(G)$. As $G$ and $T^{*}G$ are affine, we will need to understand
the global sections of this sheaf. 

We have $\Gamma(D_{h}(G))\tilde{=}O_{G}[[h]]\hat{\otimes}_{\mathbb{C}[[h]]}U_{h}(\mathfrak{g})(0)$;
where the algebra structure on the tensor product is determined by
$(f\otimes x)(g\otimes y)=hf(xg)\otimes y+fg\otimes xy$ for $f,g\in O_{G}$
and $x,y\in\mathfrak{g}$, where by $(xg)$ we mean the action of
$x$ as a left invariant vector field on $g$. Then $D_{h}(G)$ admits
both a left and a right equivariant structure for $G$, by the canonical
actions of the group on the functions and vector fields. 

We shall work with characters $\lambda\in\mathfrak{b}^{*}$, which
satisfy $\lambda(\mathfrak{n})=0$. We can apply the Hamiltonian reduction
procedure as explained above to $D_{h}(G)$ and $D_{h}(G)(0)$, where
we consider the right action of $B$ on $D_{h}(G)$, and we consider
the character $\chi=h\lambda$. In particular, the constructions here
are compatible with those of \cite{key-7}. 

From now on, $X=G/B$. We obtain a sheaf on $T^{*}(X)$, denoted $D_{h}(\lambda-\rho)$
and $D_{h}(\lambda-\rho)(0)$ (where $\rho$ denotes the sum of the
positive roots in $\mathfrak{g}$- this notation will become clear
later). The latter sheaf can also be written as follows: we can consider
the sheaf $O_{X}[[h]]\hat{\otimes}_{\mathbb{C}[[h]]}U_{h}(\mathfrak{g})(0)$
on $T^{*}X$, and we can take the quotient of this sheaf by the ideal
sheaf generated by $\{b-h\lambda(b)|b\in\mathfrak{b}\}$; by the definition
of the reduction procedure and the action of $B$, this is the same
sheaf. Under these notations, the sheaf of asymptotic differential
operators is $D_{h}(-\rho)$. We note that all the sheaves $D_{h}(\lambda)$
are $G$ equivariant with respect to the left $G$ action on $T^{*}(X)$. 

This description allows us to see that there is a ``universal''
sheaf of algebras mapping to each $D_{h}(\lambda)$; in particular,
take the quotient of $O_{X}[[h]]\hat{\otimes}_{\mathbb{C}[[h]]}U_{h}(\mathfrak{g})(0)$
by the ideal sheaf generated by the subspace $\mathfrak{n}$ of $U_{h}(\mathfrak{g})(0)$.
Then the resulting sheaf of algebras, called $D_{h}(\mathfrak{h})(0)$
can be thought of as a Hamiltonian reduction with respect to the maximal
unipotent subgroup $N$ (it is not a quantization of $T^{*}X$, but
rather of an $H$-bundle over it). The algebra $D_{h}(\mathfrak{h}):=D_{h}(\mathfrak{h})(0)[h^{-1}]$
maps (via the obvious quotient map) to each $D_{h}(\lambda).$ 

Now, we have a morphism of algebras $\Phi_{\lambda}(0):U_{h}(\mathfrak{g})(0)\to\Gamma(T^{*}X,D_{h}(\lambda)(0))$
which is defined in the obvious way using the realization of $D_{h}(\lambda)$
given above; this gives then a morphism $\Phi_{\lambda}:U_{h}(\mathfrak{g})\to\Gamma(T^{*}X,D_{h}(\lambda))$. 

Now, $U_{h}(\mathfrak{g})\tilde{=}U(\mathfrak{g})((h))$ as follows:
we have a map $U(\mathfrak{g})\to U_{h}(\mathfrak{g})$ by sending
$x\in\mathfrak{g}$ to $h^{-1}x$ in $U_{h}(\mathfrak{g})$; it is
easy to see that this map is an isomorphism onto the subalgebra generated
by $h^{-1}\mathfrak{g}$. Then we extend this map to $U(\mathfrak{g})[[h]]$
by sending $h$ to $h$ to achieve the above isomorphism. 

This allows us to relate the traditional sheaves of twisted differential
operators (as defined in \cite{key-11}) to the sheaves that we have
defined. So, let $U\subseteq X$ be an open subset, and let $V\subseteq T^{*}X$
be the inverse image of $U$ under the natural projection. Then $D_{h}(\lambda)(V)=\frac{O_{U}[[h]]\hat{\otimes}_{\mathbb{C}[[h]]}U_{h}(\mathfrak{g})(0)}{<b-h(\lambda+\rho)(b)>|_{U}}[h^{-1}]$
, while $D(\lambda)(U)=\frac{O_{U}\otimes U(\mathfrak{g})}{<b-(\lambda+\rho)(b)>}$.
So we get a map $D(\lambda)(U)\to D_{h}(\lambda)(V)$ via the above
map $U(\mathfrak{g})\to U_{h}(\mathfrak{g)}$ and the inclusion $O_{X}\to O_{X}[[h]]$. 

So, if we consider the restriction of $\Phi_{\lambda}$ to $U(\mathfrak{g})$,
we obtain a morphism $\tilde{\Phi}_{\lambda}:U(\mathfrak{g})\to\Gamma(T^{*}X,D(\lambda))$.
Now, by the results in \cite{key-11}, we have that the kernel of
$\tilde{\Phi}_{\lambda}$ is the ideal $U(\mathfrak{g})I_{\lambda}$,
where $I_{\lambda}$ is ideal in the center of $U(\mathfrak{g})$
corresponding to $\lambda$ (here we use the fact that $Spec(Z(\mathfrak{g}))=\mathfrak{h}//W$;
see \cite{key-11} for details), so that we have $\Gamma(T^{*}X,D(\lambda))\tilde{=}U(\mathfrak{g})/U(\mathfrak{g})I_{\lambda}:=U(\mathfrak{g})_{\lambda}$.
Therefore, we see that the kernel of $\Phi_{\lambda}$ contains the
ideal $J_{\lambda}=I_{\lambda}((h))$, and the kernel of $\Phi_{\lambda}(0)$
contains $J_{\lambda}(0):=J_{\lambda}\cap U_{h}(\mathfrak{g})(0)$. 

Now we wish to actually calculate the image and kernel of $\Phi_{\lambda}$.
We start by recalling a very general fact:
\begin{lem}
\label{lem:Basic-CC-lemma}Let $Z$ be an algebraic variety\footnote{We do not even require smoothness for this lemma.}
and let $O_{h}$ be a $\mathbb{C}[[h]]$-algebra on $Z$ that is an
$h$-complete, $h$-flat deformation of $O_{Z}$. Suppose that $R^{i}\Gamma(O_{Z})=0$
for all $i>0$. Then $R^{i}\Gamma(O_{h})=0$ for all $i>0$ and $\Gamma(O_{h})/h\tilde{\to}\Gamma(O_{Z})$.
Further, $\Gamma(O_{h})$ is itself $h$-flat and $h$ -complete.
\end{lem}

\begin{proof}
The vanishing of $R^{i}\Gamma(O_{h})$ follows from the vanishing
of $R^{i}\Gamma(O_{Z})$ by \cite{key-7}, lemma 2.12. Then, from
the short exact sequence 
\[
0\to O_{h}\xrightarrow{h}O_{h}\to O_{Z}\to0
\]
and the vanishing of $R^{1}\Gamma(O_{h})$, we see that $\Gamma(O_{h})/h\tilde{\to}\Gamma(O_{Z})$.
Finally, applying \cite{key-8}, corollary 1.6.2, we see that $R\Gamma(O_{h})=R\text{Hom}(O_{h},O_{h})$
is a cohomologically complete complex of $\mathbb{C}[[h]]$-modules.
Since it is concentrated in a single homological degree and $h$-torsion
free (since $O_{h}$ is), this implies that it is complete in the
$h$-adic topology by \cite{key-8}, lemma 1.5.4.
\end{proof}
From this we conclude
\begin{lem}
\label{lem:Phi-is-iso}The induced map $\Phi_{\lambda}(0):U_{h}(\mathfrak{g})(0)/J_{\lambda}(0)\to\Gamma(T^{*}X,D_{h}(\lambda)(0))$
is an isomorphism of algebras. 
\end{lem}

\begin{proof}
By definition we have $U_{h}(\mathfrak{g})(0)/h\tilde{=}\text{S}^{\bullet}(\mathfrak{g})$,
the symmetric algebra of $\mathfrak{g}$, and we have (c.f. \cite{key-1},
proposition 11.3.1) that the image of $J_{\lambda}(0)$ in $\text{S}^{\bullet}(\mathfrak{g})$
is exactly $I(N)$, the ideal of the nilpotent cone. Therefore 
\[
(U_{h}(\mathfrak{g})(0)/J_{\lambda}(0))/h\tilde{=}\text{S}^{\bullet}(\mathfrak{g})/I(N)
\]

On the other hand, the natural map $\text{S}^{\bullet}(\mathfrak{g})/I(N)\to\Gamma(O_{T^{*}X})$
is an isomorphism. Thus we see that $\Phi_{\lambda}(0)$ becomes an
isomorphism after reduction mod $h$. Now, we have that $U_{h}(\mathfrak{g})(0)/J_{\lambda}(0)$
is $h$-complete (as a quotient of a noetherian $h$-complete algebra),
and it follows readily from the definition that it is $h$-torsion-free.
The same is true of $\Gamma(T^{*}X,D_{h}(\lambda)(0))$ by the preceding
result. Thus both objects are cohomologically complete $\mathbb{C}[[h]]$-modules,
and so the complete Nakayama lemma (as in \cite{key-8}, corollary
1.5.9) now yields that $\Phi_{\lambda}(0)$ is an isomorphism. 
\end{proof}
\begin{rem}
From this lemma, we deduce immediately the version for $\Phi_{\lambda}$;
it is surjective and its kernel is precisely $J_{\lambda}$. We deduce
the surjectivity from the fact that $U_{h}(\mathfrak{g})=\bigcup_{n\geq o}h^{-n}U_{h}(\mathfrak{g})$,
$\Gamma(T^{*}X,D_{h}(\lambda))=\bigcup_{n\geq0}\Gamma(T^{*}X,h^{-n}D_{h}(\lambda))$,
and that each truncation of $\Phi_{\lambda}$ to $h^{-n}U_{h}(\mathfrak{g})$
is surjective; by exactly the same argument as in the proof of the
lemma. The identification of the kernel is proved the same way (i.e.,
look at each truncation). 
\end{rem}

Inspired by the above lemma, we introduce the ring $U_{h,\lambda}:=U_{h}(\mathfrak{g})/J_{\lambda}$.
Then we have a localization theorem for asymptotic differential operators,
as follows: 
\begin{thm}
Let $\lambda$ be an anti-dominant weight. Then $\Gamma:Mod^{qc}(D_{h}(\lambda))\to Mod(U_{h,\lambda})$
is an equivalence of categories. Further, $\Gamma$ takes coherent
$D_{h}(\lambda)$ modules to coherent $U_{h}(\mathfrak{g})$ modules,
and we have that $\Gamma:Mod^{coh}(D_{h}(\lambda))\to Mod^{coh}(U_{h,\lambda})$
is an equivalence of categories as well. 
\end{thm}

This theorem follows formally (c.f. \cite{key-1}, proposition 1.4.4)
from the following: 
\begin{thm}
\label{thm:Converativity-for-asymptotic}For $\lambda$ anti-dominant,
\textup{$\Gamma:Mod_{qc}(D_{h}(\lambda))\to Mod(U_{h,\lambda})$ is
an exact and conservative functor (i.e., $\Gamma(M)=0$ implies $M=0$). }
\end{thm}

The inverse functor is given by the localization of modules: $M\to D_{h}(\lambda)\otimes_{U_{h,\lambda}}M$.
Here $U_{h,\lambda}$ represents the constant sheaf on $T^{*}X$,
and $M$ as well. $D_{h}(\lambda)$ is a module over $U_{h,\lambda}$
via the identification of global sections. 

The proof of this theorem will follow from similar considerations
to the classical case. To begin, recall that for each $\psi\in\mathfrak{b}^{*}$
which comes from a character of $\mathfrak{h}^{*}$, we have the induced
line bundle $O_{\psi}$ on $X$. We choose a normalization so that
the anti-dominant weights correspond to ample line bundles. By abuse
of notation, we shall also denote by $O_{\psi}$ the line bundle $\pi^{*}O_{\psi}$
on $T^{*}X$. For anti-dominant $\psi$, these bundles are ample over
the base scheme $N$, as can easily be seen by looking at the morphism
to projective space over $N$ corresponding to $O_{\psi}$. Since
$N$ is an affine variety, all of Serre's theorems about ample bundles
on projective varieties go through in this case (see \cite{key-10},
chapter 3, section 5). The key to the argument will be the twisting
of $D_{h}(\lambda)$ modules by these line bundles. We formulate this
twisting by using the Hamiltonian reduction definition of differential
operators, following \cite{key-7}. 

In particular, as discussed in the previous section, we have equivalences
of categories 
\[
Mod^{qc}(D_{h}(\lambda))\tilde{\to}Mod_{\lambda}^{B,qc}(D_{h}(T^{*}G))
\]
 and 
\[
Mod^{coh}(D_{h}(\lambda))\tilde{\to}Mod_{\lambda}^{B,coh}(D_{h}(T^{*}G))
\]
On $T^{*}G$, we have, if $V$ is any finite dimensional $B$-module,
the twist functor 
\[
Mod^{B,?}(D_{h}(T^{*}G))\to Mod^{B,?}(D_{h}(T^{*}G))
\]
given by $M\to M\otimes V$ (where $?$ stands for either coherent
or quasi-coherent). In the case where $V=\mathbb{C}_{\psi}$, the
latter functor is an equivalence of categories 
\[
Mod_{\lambda}^{B,?}(D_{h}(T^{*}G))\tilde{\to}Mod_{\lambda+\psi}^{B,?}(D_{h}(T^{*}G)).
\]
Combining the two functors, we get equivalences of categories 
\[
Mod^{?}(D_{h}(\lambda))\tilde{\to}Mod^{?}(D_{h}(\lambda+\psi)),
\]
and we shall refer to this functor as $F_{\psi}$ (in both the coherent
and quasicoherent cases, with the character $\lambda$ being understood).
We can describe this functor directly as follows: we denote the quotient
morphism by $p:\mu^{-1}(0)\to T^{*}(G/B)$. Then, we define $V_{\psi}:=\mathbb{C}_{\psi}\otimes O_{T^{*}G}|_{\mu^{-1}(0)}$
, with its $B$-equivariant structure defined by the representation
$\mathbb{C}_{\psi}$. Then we have 
\begin{claim}
\label{claim:little-iso}The sheaf $p_{*}(V_{\psi})^{B}\tilde{=}O_{\psi}$. 
\end{claim}

\begin{proof}
To check this, it suffices to show that $p^{*}(O_{\psi})\tilde{=}V_{\psi}$(as
we are dealing with $B$-equivariant sheaves, and $p$ is a $B$-principal
bundle morphism). To check this, it suffices to take the line bundle
$O_{\psi}$ on $G/B$, pull back to $G$, and then pull back to $\mu^{-1}(0)$.
But the pullback of $O_{\psi}$ to $G$ is the sheaf $\mathbb{C}_{\psi}\otimes O_{G}$,
by the definition of the induced bundle. This proves the claim. 
\end{proof}
So, given a module $M$(0) in $Mod_{\lambda}^{B,qc}(D_{h}(\lambda)(0))$,
our functor $F_{\psi}$ is given by 
\[
M(0)\to p_{*}Hom_{D_{h}(T^{*}G)(0)}(L_{\lambda}(0),L_{\lambda}(0)\otimes_{p^{-1}D_{h}(\lambda)(0)}p^{-1}M(0)\otimes\mathbb{C}_{\psi})^{B}
\]
Thus, we see that 
\[
F_{\psi}(M(0))/hF_{\psi}(M(0))\tilde{=}M(0)/hM(0)\otimes_{O_{T^{*}X}}O_{\psi}
\]
using the fact that the functor 
\[
M(0)\to p_{*}Hom_{D_{h}(T^{*}G)(0)}(L_{\lambda}(0),L_{\lambda}(0)\otimes_{p^{-1}D_{h}(\lambda)(0)}p^{-1}M(0))
\]
 is just the identity (see above). 

We can also consider the twist functor in the case of the $B$ module
$L(\nu)$, where $L(\nu)$ is the irreducible $G$-module of highest
weight $\nu$ (where $\nu$ is supposed to be dominant integral).
This gives a functor $G_{\nu}:Mod^{?}(D_{h}(\mathfrak{h})(0))\to Mod^{?}(D_{h}(\mathfrak{h})(0))$
(and, of course, a $G_{\nu}:Mod^{?}(D_{h}(\mathfrak{h}))\to Mod^{?}(D_{h}(\mathfrak{h}))$),
which, however, does not map a subcategory of the type $Mod^{?}(D_{h}(\lambda)(0))$
to another, because the module $L(\nu)$ does not have a $B$-character. 

On the other hand, we have, by standard weight theory, a finite \textbf{$B$-}filtration
of $L(\nu)$, $\{L(\nu)_{i}\}$, such that the subquotients $L(\nu)_{i}/L(\nu)_{i-1}$
are one dimensional $B$-modules. We now let $M(0)\in Mod^{qcoh}(D_{h}(\lambda)(0))$.
If we twist $M$ by $L(\nu)$; then the result is 
\[
M(0)\to p_{*}Hom_{D_{h}(T^{*}G)(0)}(L_{\lambda}(0),L_{\lambda}(0)\otimes_{p^{-1}D_{h}(\lambda)(0)}p^{-1}M(0)\otimes L(v))^{B}:=G_{v}(M(0))
\]

Now, because the module $L(\nu)$ has a $G$-action, the sheaf $p_{*}(L(\nu)\otimes O_{T^{*}G}|_{\mu^{-1}(0)})^{B}$
is actually a trivial vector bundle over $T^{*}X$. So in this case
we conclude that $G_{\nu}(M(0))\tilde{=}M(0)\otimes L(\nu)$; i.e.,
it is simply a finite direct sum of copies of $M(0)$. 

Then, we have a filtration on $G_{\nu}(M(0))$, $\{G_{\nu}(M(0))_{i}\}$,
induced from that on $L(v)$; this is a filtration of $D_{h}(\mathfrak{h})(0)-$modules.
The important point is the following: the subquotients of this filtration
$G_{\nu}(M(0))_{i}/G_{\nu}(M(0))_{i-1}$ are isomorphic to the sheaf
\[
p_{*}Hom_{D_{h}(T^{*}G)(0)}(L_{\lambda}(0),L_{\lambda}(0)\otimes_{p^{-1}D_{h}(\lambda)(0)}p^{-1}M(0)\otimes\mathbb{C}_{\nu_{i}})^{B}\tilde{=}F_{\nu_{i}}(M(0)).
\]

And, of course, the same isomorphism holds after inverting $h$ everywhere. 

Now, if we restrict our attention to the copy of $U(\mathfrak{g})$
described above (the one generated by elements of the form $h^{-1}x$
for $x\in\mathfrak{g}$), then we have that the ideal $I_{\lambda+\nu_{i}}$
acts trivially on $F_{\nu_{i}}(M)$. If we associate to each $I_{\lambda}$
the central character $\chi_{\lambda}$, then we have that for all
$\xi\in Z(\mathfrak{g})$, the product $\Pi_{i}(\xi-\chi_{\lambda+\nu_{i}}(\xi))$
annihilates $G_{\nu}(M)$. Therefore, we can write $G_{\nu}(M)\tilde{=}\bigoplus G_{\nu}(M)_{[\psi]}$
a direct sum of generalized $Z(\mathfrak{g})$-eigensheaves. 

Repeating the proof of \cite{key-11} (lemma 1,pg 24) verbatim, we
can conclude the following 
\begin{lem}
\label{lem:basic-weight-lemma}Let $\lambda$ be an anti-dominant
weight, and $\mu$ a dominant integral weight and let $M\in Mod^{qc}(D_{h}(\lambda))$.
Then we have that, $M\tilde{=}G_{\mu}F_{-\mu}(M)_{[\lambda]}$. 

Therefore, we see that $M$ is a direct summand of $F_{-\mu}(M)\otimes L(\mu)$. 

Further, let $w_{0}$ denote the longest element of the Weyl group.
Then the sheaf $F_{-\mu}(M)$ is a direct summand of $G_{-w_{0}\mu}(M)\tilde{=}M\otimes L(-w_{0}\mu)$. 
\end{lem}

Now we can give the 
\begin{proof}
(of \ref{thm:Converativity-for-asymptotic}). We first handle
exactness. We note that any $M\in Mod^{qc}(D_{h}(\lambda))$ is a
direct limit of coherent $D_{h}(\lambda)$- modules (see section 2)
and that cohomology commutes with direct limits on a noetherian space.
So WLOG $M\in Mod^{coh}(D_{h}(\lambda))$ with $\lambda$ anti-dominant,
and with $M(0)$ a lattice. Then $M(0)/hM(0)$ is a coherent sheaf
on $T^{*}X$, and by Serre's theorem, there exists $\mu>>0$ so that
$H^{i}(O_{T^{*}X},M(0)/hM(0)\otimes O_{-\mu})=0$ for all $i>0$.
Further, we know that $M(0)/hM(0)\otimes O_{-\mu}\tilde{=}F_{-\mu}(M(0))/hF_{-\mu}(M(0))$. 

Now, by \ref{lem:basic-vanishing}, we have that $H^{i}(T^{*}X,F_{-\mu}(M(0)))=0$
for all $i>0$. Therefore we conclude that $H^{i}(T^{*}X,F_{-\mu}(M))=0$
for all $i>0$ as $F_{-\mu}(M)=F_{-\mu}(M(0))[h^{-1}]$. 

But now, by \ref{lem:basic-weight-lemma}, we have an injection
$$H^{i}(T^{*}X,M)\to H^{i}(T^{*}X,F_{-\mu}(M)\otimes L(\mu))=H^{i}(T^{*}X,F_{-\mu}(M))\otimes L(\mu)=0$$
for all $i>0.$ So exactness is shown.

We now show that $\Gamma(T^{*}X,M)=0$ implies $M=0$, for $M\in Mod^{qc}(D_{h}(\lambda))$.
Our assumption is that we have that $0=\Gamma(T^{*}X,M(0))[h^{-1}]$,
which implies that for each global section $s$, there exists some
$n\geq1$ such that $h^{n}s=0$. Now, we define, for each $i\geq1$,
the subsheaf $M(0)_{i}$, which is the sheaf of local sections of
$M(0)$ which are annihilated by $h^{i}$. Then the theorem becomes
equivalent to showing $M(0)=\cup_{i}M(0)_{i}$. If not, we consider
the quotient sheaf $N(0)=M(0)/\cup_{i}M(0)_{i}$. Then $N(0)$ is
a nontrivial $D_{h}(\lambda)(0)$-module. 

We note that by definition, $M\tilde{=}N:=N(0)[h^{-1}]$. Therefore,
$\Gamma(T^{*}X,N)=0$. Further, the construction of $N(0)$ implies
that the natural map $N(0)\to N$ is injective (i.e., there are no
local sections that are killed by a power of $h$). So we see that
it suffices to show $N=0$. 

Now, there exists some dominant $\mu$ such that $\Gamma(T^{*}X,F_{-\mu}(N(0))/hF_{-\mu}(N(0)))\neq0$
by Serre's theorem's about ample line bundles (we note that the assumption
that $N(0)\neq0$ implies $N(0)/hN(0)\neq0$ by the Nakayama lemma).
This implies (by \ref{lem:basic-vanishing}) that $\Gamma(T^{*}X,F_{-\mu}(N(0)))\neq0$.
In turn, the module $F_{-\mu}(N(0))$ injects to $F_{-\mu}(N)$ as
$F_{-\mu}(N(0))$ has no local sections which are killed by a power
of $h$ (this follows from the corresponding fact about $N(0)$).

Now, let $w_{0}$ denote the longest element of the Weyl group. Given
a dominant weight $\mu$, we have an injection $F_{-\mu}(N)\to G_{-w_{0}\mu}(N)$
(see lemma 5.7). Therefore , the fact that $0\neq\Gamma(T^{*}X,F_{-\mu}(N))$
implies 

\[
0\neq\Gamma(T^{*}X,G_{-w_{0}\mu}(N))=\Gamma(T^{*}X,N)\otimes L(-w_{0}\mu)=0,
\]
a contradiction. 
\end{proof}
We now discuss the $\mathbb{C}^{*}$- equivariance conditions that
need to be imposed. The above theorem deals with categories of modules
defined over the field $\mathbb{C}((h))$, whereas the original localization
theorem deals with the $\mathbb{C}$-linear category of $U(\mathfrak{g})_{\lambda}$-modules.
We now show how to recover the original theorem from the one above. 

First of all, we have canonical $\mathbb{C}^{*}$- actions on both
$U_{h,\lambda}$ and $D_{h}(\lambda)$: for $U_{h,\lambda}$we let
$\phi_{t}(h)=t^{-1}h$ and $\phi_{t}(g)=tg$ for all $g\in\mathfrak{g}$.
This is the standard $\mathbb{C}^{*}$ action on $U_{h}(\mathfrak{g})$
and it induces one on $U_{h,\lambda}$. For $D_{h}(\lambda)$ we start
with the sheaf $D_{h}$ on $T^{*}G$, and consider the action of $\mathbb{C}^{*}$
by dilation of the fibers. $D_{h}$ is equivariant with respect to
this action by setting $\psi_{t}(h)=t^{-1}h$, $\psi_{t}(\xi)=t\xi$
where $\xi$ is any global vector field. It is easy to observe that
this action preserves the set $\mu^{-1}(0)\subseteq T^{*}G$ and commutes
with the action of $B$ on the right. Thus we see that this gives
rise to a $\mathbb{C}^{*}$-action on $T^{*}X$ with respect to which
all the sheaves $D_{h}(\lambda)$ are equivariant. 

Now we can make some observations about these actions: first, $U_{h,\lambda}^{\mathbb{C}^{*}}\tilde{=}U(\mathfrak{g})_{\lambda}$.
This follows from the fact that $U_{h}(\mathfrak{g})^{\mathbb{C}^{*}}\tilde{=}U(\mathfrak{g})$,
which is simply the identification of $U(\mathfrak{g})$ with the
subalgebra of $U_{h}(\mathfrak{g})$ generated by $h^{-1}\mathfrak{g}$
(that these are the $\mathbb{C}^{*}$-fixed elements follows immediately
from the description of the $\mathbb{C}^{*}$ action given above). 

Next, we can observe that, for an open subset $U\subseteq G/B$, if
$V=\pi^{-1}(U)$, we have $D_{h}(\lambda)(V)^{\mathbb{C}^{*}}\tilde{=}D(\lambda)(U)$
(by the same reasoning as the above). We can in fact make the stronger
statement that we have an equivalence of categories: $Mod^{coh,\mathbb{C}^{*}}(D_{h}(\lambda))\tilde{\to}Mod^{coh}(D(\lambda))$,
where the left-hand side denotes the category of $\mathbb{C}^{*}$-equivariant
coherent $D_{h}(\lambda)$ modules. This equivalence is given by taking
$\mathbb{C}^{*}$-invariant sections. 

Given all this, the statement of the final theorem (the original Beilinson-Bernstein
localization) is intuitively clear: 
\begin{thm}
\label{thm:Loc-on-Flag} For $\lambda$ anti-dominant, we have an
equivalences of categories: 
\[
Mod^{f.g.}(U(\mathfrak{g})_{\lambda})\tilde{\to}Mod^{coh,\mathbb{C}^{*}}(D_{h}(\lambda))\tilde{\to}Mod^{coh}(D(\lambda))
\]

\[
Mod(U(\mathfrak{g})_{\lambda})\tilde{\to}Mod^{qc,\mathbb{C}^{*}}(D_{h}(\lambda))\tilde{\to}Mod^{qc}(D(\lambda))
\]
\end{thm}

\begin{proof}
This proof follows the mechanics of the previous argument (which we
will use along the way). In one direction, we have the functor $M\to\Gamma(M)^{\mathbb{C}^{*}}$
which takes $\mathbb{C}^{*}$-equivariant coherent $D_{h}(\lambda)$-
modules to $U(\mathfrak{g})_{\lambda}$ modules. We wish to show that
its image lives inside the category of finitely generated $U(\mathfrak{g})_{\lambda}$
modules. (This argument is more or less standard, but the presence
of the $\mathbb{C}^{*}$-action requires some care). To do so, we
first to show that $\Gamma^{\mathbb{C}^{*}}$ is an exact and conservative
functor. The exactness is clear from the exactness of $\Gamma$ as
taking invariants for a $\mathbb{C}^{*}$-action is exact. To show
that it is conservative, we again only need to show that taking $\mathbb{C}^{*}$
invariants is conservative; this follows from our discussion of $\mathbb{C}^{*}$-actions
in \ref{lem:C-star-actions} (we note that the discussion goes
through in this case, as we are taking the invariants functor on the
category of $U_{h,\lambda}$-modules). Therefore, we conclude that
every $\mathbb{C}^{*}$-equivariant coherent $D_{h}(\lambda)$-module
$M$ is generated by $\mathbb{C}^{*}$-invariant global sections:
let $N$ be the sub-$D_{h}(\lambda)$-module of $M$ generated by
the $\mathbb{C}^{*}$-invariant global sections. Then we have the
exact sequence $0\to N\to M\to M/N\to0$, applying our exact functor
shows that $\Gamma(M/N)^{\mathbb{C}^{*}}=0$, so $M/N$ is $0$ as
required. 

To complete the argument about finite generation, we note that our
module $M$ is locally finitely generated: for any affine open covering
of $T^{*}X$, $\{U_{i}\}$, we have that $M|_{U_{i}}$ is a finitely
generated $D_{h}(\lambda)|_{U_{i}}$-module. Now, we choose an affine,
open, finite $\mathbb{C}^{*}$-invariant cover of $T^{*}X$ (one can
always do this for a normal variety with a $\mathbb{C}^{*}$-action,
although in this case it is obvious as we can just take an affine
cover of $X$ and pull back to $T^{*}X$). Then for each $M|_{U_{i}}$,
we have that $(M|_{U_{i}})^{\mathbb{C}^{*}}$ is finitely generated
as a $D_{h}(\lambda)^{\mathbb{C}^{*}}|_{U_{i}}$-module by \ref{lem:C-star-actions}.
By the above, we can choose finitely many $\mathbb{C}^{*}$-invariant
global sections that restrict to generators of $(M|_{U_{i}})^{\mathbb{C}^{*}}$.
By the finiteness of the cover, we have found finitely many global
sections which generate the $D_{h}(\lambda)^{\mathbb{C}^{*}}$-module
$M^{\mathbb{C}^{*}}$. Therefore, these elements generate the $U(\mathfrak{g})$-module
$\Gamma(M)^{\mathbb{C}^{*}}$. 

Now, the functor in the opposite direction is given by $V\to D_{h}(\lambda)\otimes_{U(\mathfrak{g})_{\lambda}}V$.
This is clearly a (quasi)coherent, $\mathbb{C}^{*}$-equivariant $D_{h}(\lambda)$-module
(with the $\mathbb{C}^{*}$-action given via the one on $D_{h}(\lambda)$).
Now the proof that these two functors are inverse is totally standard. 
\end{proof}
Our goal in the rest of this section is to explain how localization
works when one replaces the usual $\mathbb{C}^{*}$-action with the
action that one needs to study the finite $W$-algebras. We note that
the above proof doesn't depend on the particular $\mathbb{C}^{*}$-action,
but that both the algebra of invariants and the sheaf of invariant
operators do. 

Fix a nilpotent element $e\in\mathfrak{g}$. As recalled above in
\ref{sec:W-algebras-and-quantum}, there is a natural homomorphism
$\gamma:\mathbb{C}^{*}\to G$, which leads to a homomorphism $\rho:\mathbb{C}^{*}\to\mathbb{C}^{*}\times G$,
defined as 
\[
\rho(t)=(t^{-2},\gamma(t))
\]
Now, for any $\lambda$ as above, the sheaf $\mathcal{D}_{h}(\lambda)$
is equivariant over $\mathbb{C}^{*}\times G$, where the $\mathbb{C}^{*}$
action is the one described above, and the $G$ action is the natural
one on differential operators. Therefore, we may restrict this action
to $\mathbb{C}^{*}$ via $\rho$ to obtain our new $\mathbb{C}^{*}$
action on $D_{h}(\lambda)$. 

Via the adjoint action of $G$ on $\mathfrak{g}$, and the natural
$\mathbb{C}^{*}$ action on $U_{h}(\mathfrak{g})$ via the grading,
we obtain an action of $\mathbb{C}^{*}\times G$ on $U_{h}(\mathfrak{g})$.
Again restricting via $\rho$ we obtain our required $\mathbb{C}^{*}$-action
on $U_{h}(\mathfrak{g})$. It is worth describing this action explicitly:
we have the decomposition $\mathfrak{g}=\oplus\mathfrak{g}(i)$ which
was the weight decomposition for our chosen $\mathfrak{sl}_{2}$-triple.
Then for $g\in\mathfrak{g}(i)$, we put $\sigma_{t}(g)=t^{i+2}g$,
and we let $\sigma_{t}(h)=t^{2}h$, and extend this to all of $U_{h}(\mathfrak{g})$
in the natural way. This corresponds to the Kazhdan filtration on
$U(\mathfrak{g})$. 

Because $h$ has degree 2, we work from now on with the extended ring
$U_{h}(\mathfrak{g})\otimes_{\mathbb{C}((h))}\mathbb{C}((h^{1/2}))$,
and we similarly extend the sheaf $D_{h}(\lambda)$. 
\begin{lem}
\label{lem:This-action-preserves}This action preserves the ideal
$J_{\lambda}\subset U_{h}(\mathfrak{g})$. 
\end{lem}

\begin{proof}
To show this, we describe a generating set for $J_{\lambda}$ as follows:
the Killing form is a perfect pairing between $\mathfrak{g}(i)$ and
$\mathfrak{g}(-i)$. We choose bases in these spaces which are dual
to each other; this then gives a basis of $\mathfrak{g}$, for a basis
element $X_{i}$ we let $\tilde{X}_{i}$ denote its dual element.
Let $\phi$ be any finite dimensional representation of $\mathfrak{g}$.
According to \cite{key-9} (prop 5.32, proof of theorem 5.44), a generating
set for the ideal $I_{\lambda}\subseteq U(\mathfrak{g})$ is given
by elements of the form $\sum_{i_{1},...,i_{n}}Tr\phi(X_{i_{i}}***X_{i_{n}})(\tilde{X}_{i_{1}}***\tilde{X}_{i_{n}})-\sum_{i_{1},...,i_{n}}Tr\phi(X_{i_{i}}***X_{i_{n}})(\lambda(\tilde{X}_{i_{1}})***\lambda(\tilde{X}_{i_{n}}))$.
Therefore, we conclude that a generating set for $J_{\lambda}$ is
given by $\sum_{i_{1},...,i_{n}}h^{-n}Tr\phi(X_{i_{i}}***X_{i_{n}})(\tilde{X}_{i_{1}}***\tilde{X}_{i_{n}})-\sum_{i_{1},...,i_{n}}Tr\phi(X_{i_{i}}***X_{i_{n}})(\lambda(\tilde{X}_{i_{1}})***\lambda(\tilde{X}_{i_{n}}))$.
Now, the only way that $Tr\phi(X_{i_{1}}***X_{i_{n}})$ can be nonzero
is if, letting $X_{i_{k}}\in\mathfrak{g}(j_{k})$, $\sum_{i=1}^{n}j_{k}=0$:
this follows from the fact that the representation $\phi$ inherits
a grading from the same $\mathfrak{sl}_{2}$-action; and any matrix
that shifts the grading non-trivially is traceless. Now, since $\tilde{X}_{i_{k}}$
lives in degree $-j_{k}$, it must be that the element $\tilde{X}_{i_{1}}***\tilde{X}_{i_{n}}$
also has degree $0$ with respect to this $\mathfrak{sl}_{2}$ action.
By the definition of the $\mathbb{C}^{*}$-action we are working with,
we see that $h^{-(j_{k}+2)/2}\tilde{X}_{i_{k}}$ is $\mathbb{C}^{*}$-invariant,
and so it follows that the generating set considered above is in fact
$\mathbb{C}^{*}$-invariant; and so, therefore, is the ideal $J_{\lambda}$. 
\end{proof}
We consider now the ring of invariants with respect to this action.
Clearly, this ring consists of series, infinite in positive powers
of $h$, whose terms are products of elements of the form $h^{-(i+2)/2}g$
with $g\in\mathfrak{g}(i)$. Therefore, this ring is not isomorphic
to the enveloping algebra $U(\mathfrak{g})$. In particular, it will
include infinite series whose terms come from $\oplus_{i\leq-3}\mathfrak{g}(i)$
(which, we note, is a subalgebra of $\mathfrak{m}_{l}$), and in fact,
it is clear that this algebra is the completion of $U(\mathfrak{g})$
with respect to the nilpotent Lie subalgebra $\oplus_{i\leq-3}\mathfrak{g}(i)$
(one can consult \cite{key-4} section 5 for details on this notion
of completion; however, we will not use this). Therefore, it follows
from our computation of the global sections of $D_{h}(\lambda)$ above
that $\Gamma(T^{*}X,D_{h}(\lambda)^{\mathbb{C}^{*}})\tilde{=}U_{h}(\mathfrak{g})^{\mathbb{C}^{*}}/J_{\lambda}\cap U_{h}(\mathfrak{g})^{\mathbb{C}^{*}}$.
To quantify this, we consider the copy of $U(\mathfrak{g})\subseteq U_{h}(\mathfrak{g})^{\mathbb{C}^{*}}$
(just the algebra generated by $h^{-(i+2)/2}g$ for $g\in\mathfrak{g}(i)$);
and we note that $J_{\lambda}\cap U_{h}(\mathfrak{g})^{\mathbb{C}^{*}}$
is generated by the elements given in the proof of \ref{lem:This-action-preserves},
which are simply generators for the ideal $I_{\lambda}\subseteq U(\mathfrak{g})$.
So the ideal $J_{\lambda}\cap U_{h}(\mathfrak{g})^{\mathbb{C}^{*}}$
is the ideal in $U_{h}(\mathfrak{g})^{\mathbb{C}^{*}}$ generated
by $I_{\lambda}$. 

With this in hand, we can repeat verbatim the proof of \ref{thm:Loc-on-Flag}
and obtain
\begin{thm}
For $\lambda$ anti-dominant, we have equivalences of categories 
\[
Mod^{f.g.}(U_{h}(\mathfrak{g})^{\mathbb{C}^{*}}/J_{\lambda}\cap U_{h}(\mathfrak{g})^{\mathbb{C}^{*}})\tilde{\to}Mod^{coh,\mathbb{C}^{*}}(D_{h}(\lambda))
\]
\[
Mod(U_{h}(\mathfrak{g})^{\mathbb{C}^{*}}/J_{\lambda}\cap U_{h}(\mathfrak{g})^{\mathbb{C}^{*}})\tilde{\to}Mod^{qc,\mathbb{C}^{*}}(D_{h}(\lambda))
\]
 
\end{thm}

On the face of it, this theorem is not very useful, because of our
lack of knowledge of the category appearing on the left. However,
this category becomes quite tractable after one additional modification:
we have the adjoint action of the group $M_{l}$ on the algebra $U_{h}(\mathfrak{g})^{\mathbb{C}^{*}}/J_{\lambda}\cap U_{h}(\mathfrak{g})^{\mathbb{C}^{*}}$,
and we can consider the category $Mod_{\chi}^{M_{l},f.g.}(U_{h}(\mathfrak{g})^{\mathbb{C}^{*}}/J_{\lambda}\cap U_{h}(\mathfrak{g})^{\mathbb{C}^{*}})$
of $\chi$-twisted $M_{l}$-equivariant finitely generated modules.
It is easy to see that this is just the category of modules $V$ such
that for all $m\in\mathfrak{m}_{l}$, $m-\chi(m)$ acts locally nilpotently
on $V$. Now, by definition, $\chi|_{\oplus_{i\leq-3}\mathfrak{g}(i)}=0$.
Therefore, for a module in $Mod_{\chi}^{M_{l},f.g.}(U_{h}(\mathfrak{g})^{\mathbb{C}^{*}}/J_{\lambda}\cap U_{h}(\mathfrak{g})^{\mathbb{C}^{*}})$,
all of the infinite series in the ring $U_{h}(\mathfrak{g})^{\mathbb{C}^{*}}/J_{\lambda}\cap U_{h}(\mathfrak{g})^{\mathbb{C}^{*}}$
simply act via finitely many terms. Therefore, combining this with
our above description of $J_{\lambda}\cap U_{h}(\mathfrak{g})^{\mathbb{C}^{*}}$,
we see that we have a canonical equivalence of categories 
\[
Mod_{\chi}^{M_{l},f.g.}(U_{h}(\mathfrak{g})^{\mathbb{C}^{*}}/J_{\lambda}\cap U_{h}(\mathfrak{g})^{\mathbb{C}^{*}})\tilde{\to}Mod_{\chi}^{M_{l},f.g.}(U(\mathfrak{g})_{\lambda})
\]

So, a localization theorem for this category would need to consider
$M_{l}$-equivariant $D_{h}(\lambda)$-modules. It is clear, by restriction
of the $G$-action, that there is an $M_{l}$-equivariant structure
on the algebra $D_{h}(\lambda)$. Unfortunately, the above $\mathbb{C}^{*}$-action
and the $M_{l}$ action do not commute. However, we can express the
structure we want by looking at the adjoint action of the one parameter
group $\mathbb{C}^{*}$ on the group $M_{l}$, via the morphism $\gamma(t)$.
This allows us to form the semi-direct product $M_{l}\rtimes\mathbb{C}^{*}$.
Then, adding an $M_{l}$-equivariance condition to the category on
the right of the above theorem is the same as looking $D_{h}(\lambda)$-modules
that are equivariant with respect to $M_{l}\rtimes\mathbb{C}^{*}$.
If we consider $\chi$ as a map $\chi:\mathfrak{m}\to\mathbb{C}[[h]]$,
then we can consider the category of $M_{l}\rtimes\mathbb{C}^{*}$-equivariant
modules over $D_{h}(\lambda)$ so that the $M_{l}$ action has twist
$\chi$, in the sense of \ref{sec:Hamiltonian-reduction}. 

For any module in $M\in Mod_{\chi}^{M_{l}\rtimes\mathbb{C}^{*},coh}(D_{h}(\lambda))$,
its $\mathbb{C}^{*}$-invariant global sections are a module over
$U_{h}(\mathfrak{g})^{\mathbb{C}^{*}}$. The condition that $M$ be
$M_{l}$-equivariant with twist $\chi$ ensures that $\Gamma(M)^{\mathbb{C}^{*}}$
admits an $M_{l}$-action which integrates the $\chi$-twisted action
of $\mathfrak{m}_{l}\subset U_{h}(\mathfrak{g})^{\mathbb{C}^{*}}$,
where the embedding is chosen as above; i.e., taking $m\in\mathfrak{m}_{l}$
to $h^{-(i+2)/2}m$. So, combining the above observations with the
$D_{h}(\lambda)$-affineness of $T^{*}X$ gives: 
\begin{thm}
\label{thm:Localization-with-actions}For $\lambda$ anti-dominant,
we have equivalences of categories 
\[
Mod_{\chi}^{M_{l},f.g.}(U(\mathfrak{g})_{\lambda})\tilde{\to}Mod_{\chi}^{M_{l}\rtimes\mathbb{C}^{*},coh}(D_{h}(\lambda))
\]
\[
Mod_{\chi}^{M_{l}}(U(\mathfrak{g})_{\lambda})\tilde{\to}Mod_{\chi}^{M_{l}\rtimes\mathbb{C}^{*},qc}(D_{h}(\lambda))
\]
 
\end{thm}

\section{Localization for W-Algebras and the Skryabin Equivalence}

In this section we weave together the major threads of the paper and
prove the main results. We begin by applying the Hamiltonian reduction
formalism of \ref{sec:Hamiltonian-reduction} to the case of
semisimple Lie algebras and flag varieties. 

We let $O_{h,\mathfrak{g}^{*}}$ denote the sheafification of $U_{h}(\mathfrak{g})(0)$
over the variety $\mathfrak{g}^{*}$. This is a quantization of a
smooth Poisson variety, which admits an action of $G$. In particular,
after fixing a nilpotent element $e\in\mathfrak{g}$, we obtain the
action of $M_{l}$ on $O_{h,\mathfrak{g}^{*}}$; as detailed in the
previous section, we have also an action of $M_{l}\rtimes\mathbb{C}^{*}$
The natural map $U_{h}(\mathfrak{m}_{l})(0)\to U_{h}(\mathfrak{g})(0)$
upon sheafification gives a comoment map $U_{h}(\mathfrak{m}_{l})(0)\to O_{h,\mathfrak{g}^{*}}$.
Consider the character on $U_{h}(\mathfrak{m}_{l})(0)$ determined
by $\chi:\mathfrak{m}_{l}\to\mathbb{C}$, which we will also denote
it by $\chi$. Therefore we can consider $M_{l}\rtimes\mathbb{C}^{*}$-equivariant
modules over $O_{h,\mathfrak{g}^{*}}$ whose $M_{l}$-action has twist
$\chi$. We have 
\begin{lem}
The functor $\Gamma^{\mathbb{C}^{*}}$ induces an equivalence of categories
\[
Mod_{\chi}^{M_{l}\rtimes\mathbb{C}^{*},coh}(O_{h,\mathfrak{g}^{*}}[h^{-1}])\tilde{\to}Mod_{\chi}^{M_{l},f.g.}(U(\mathfrak{g}))
\]
The analogous result holds upon replacing coh with qcoh on the left
and finitely generated modules with all modules on the right. 
\end{lem}

\begin{proof}
As $\mathfrak{g}^{*}$ is an affine variety, the global section functor
is an equivalence from $Mod^{coh}(O_{h,\mathfrak{g}^{*}})$ to $Mod^{f.g.}(U_{h}(\mathfrak{g})(0))$,
and the same holds upon inverting $h$. Let $M(0)$ be an $M_{l}\rtimes\mathbb{C}^{*}$-equivariant
module over $U_{h}(\mathfrak{g})(0)$. Then, as explained above, $M(0)[h^{-1}]^{\mathbb{C}^{*}}$is
a module over a certain completion of $U(\mathfrak{g})$, along an
ideal generated by a Lie-subalgebra of $\mathfrak{m}_{l}$. We thus
have the action of $\mathfrak{m}_{l}$ on $M(0)[h^{-1}]^{\mathbb{C}^{*}}$,
and the condition that $M(0)$ has twist $\chi$ ensures that the
$\chi$-twisted action of $M_{l}$ on $M(0)[h^{-1}]^{\mathbb{C}^{*}}$
integrates to an action of $M_{l}$.

As $M_{l}$ is a unipotent group, this implies that $m-\chi(m)$ acts
locally nilpotently on $M(0)[h^{-1}]^{\mathbb{C}^{*}}$, which implies
that this module is actually a $U(\mathfrak{g})$-module, and so we
see $\Gamma^{\mathbb{C}^{*}}$ takes values in the correct category. 

To obtain the inverse functor, we note that any module in $Mod_{\chi}^{M_{l},f.g.}(U(\mathfrak{g}))$
is necessarily locally finite with respect to the action of $m-\chi(m)$
and so the $U(\mathfrak{g})$-action extends uniquely to the structure
of a module over $U_{h}(\mathfrak{g})^{\mathbb{C}^{*}}$, and therefore
it is the $\mathbb{C}^{*}$-invariants of a unique coherent $\mathbb{C}^{*}$-equivariant
module over $U_{h}(\mathfrak{g})$, which necessarily has an action
of $M_{l}$ with twist $\chi$. Localizing this module over $\mathfrak{g}^{*}$
yields the required inverse functor. 
\end{proof}
Now we apply the machinery of \ref{sec:Hamiltonian-reduction}
\begin{prop}
\label{prop:Ham-reduction-on-g} The assumptions of \ref{prop:Reduction-equivalence}
are satisfied for the action of $M_{l}$ on $O_{h,\mathfrak{g}^{*}}$
with respect to the character $\chi$. Therefore we have a quantization
$O_{h,S}$ of $S$ given by Hamiltonian reduction and an equivalence
of categories 
\[
Mod_{\chi}^{M_{l}\rtimes\mathbb{C}^{*},coh}(O_{h,\mathfrak{g}^{*}}[h^{-1}])\tilde{\to}Mod^{\mathbb{C}^{*},coh}(O_{h,S}[h^{-1}])
\]
Furthermore, there is an equivalence of categories 
\[
Mod^{\mathbb{C}^{*},coh}(O_{h,S}[h^{-1}])\tilde{\to}Mod^{f.g.}(U(\mathfrak{g},e))
\]
The analogous statements hold when coh is replaced by qcoh.
\end{prop}

We will prove this shortly. Note, however, that, when combined with
the previous result, it gives the Skryabin equivalence 
\begin{cor}
There is an equivalence of categories 
\[
Mod_{\chi}^{M_{l}}(U(\mathfrak{g}))\tilde{\to}Mod(U(\mathfrak{g},e))
\]
 given by $V\to V^{M_{l}}$. It induces an equivalence on the finitely
generated modules on each side.
\end{cor}

In fact, it is worth noting here that the group $M_{l}$ is unnecessary
in the category on the left: because it is a connected unipotent group,
any module which is locally nilpotent for $\mathfrak{m}_{l}$ (with
respect to character$\chi$) will carry a unique action of $M_{l}$.
However, it does make the statement of the functor cleaner. 
\begin{proof}
(of \ref{prop:Ham-reduction-on-g}) The proof of the first statement
is rather similar in spirit to the proof of \ref{lem:Good-Conditions-For-T^*}.
Although there is no formal Darboux theorem for Poisson varieties,
we do have that the Poisson variety $\mathfrak{g}^{*}$ is an affine
space, and the Poisson structure comes from a linear Poisson bracket.
The moment map $\mathfrak{g}^{*}\to\mathfrak{m}_{l}^{*}$ is simply
the dual to the inclusion, and the space $\mu^{-1}(\chi)$ is isomorphic
to $S\times M_{l}$ as an $M_{l}$-variety. 

Let $x\in\mu^{-1}(\chi)$. By the above discussion the formal completion
of $O_{h,\mathfrak{g}^{*}}$ at the point $x$, denoted $\widehat{O}_{h,\mathfrak{g}^{*}}$,
is isomorphic to the completion of $U_{h}(\mathfrak{g})(0)$ at the
ideal generated by $h$ and $\mathfrak{g}$. We may therefore choose
coordinates, denoted $\{x_{i},y_{i},z_{i}\}$ that are a basis of
$\mathfrak{g}$ such that the $x_{i}$ span $\mathfrak{m}_{l}$, and
the $z_{i}$ correspond to the dual of $\mathfrak{m}_{l}$ under the
Killing form. Elements of $\widehat{O}_{h,\mathfrak{g}^{*}}$ are
then uniquely represented as series 
\[
\sum_{i,I,J,K}a_{i,I,J,K}h^{i}x^{I}y^{J}z^{K}
\]
where $a_{i,I,J,K,}\in\mathbb{C}$ (here $I,J,K$ represent multi-indices).
The image of the comoment map in this formal completion is then simply
the set of series for which each term has a nonzero contribution of
$x^{I}$. Thus the quotient by this ideal is $h$-torsion free, and
the set of $M_{l}$-invariants in this quotient is $h$-torsion free
and quantizes a formal neighborhood of a point in $S$. It follows
that the first two conditions of \ref{prop:Reduction-equivalence}
are satisfied. For the final condition (the Ext vanishing), we have
that the $x_{i}$ and the $z_{i}$ are dual bases with respect to
the Poisson bracket. This allows one to again reduce the question
to the vanishing of de Rham cohomology for affine space (exactly as
in \cite{key-3}, proposition 5.1, or the very similar argument in
\cite{key-25}, lemma 4.14) and the result follows. 

Finally let us give the equivalence 
\[
Mod^{\mathbb{C}^{*},coh}(O_{h,S}[h^{-1}])\tilde{\to}Mod^{f.g.}(U(\mathfrak{g},e))
\]
As $S$ is affine we have 
\[
Mod^{\mathbb{C}^{*},coh}(O_{h,S}[h^{-1}])\tilde{\to}Mod^{\mathbb{C}^{*},f.g.}(\Gamma(O_{h,S}[h^{-1}]))\tilde{\to}Mod^{f.g.}(\Gamma(O_{h,S}[h^{-1}])^{\mathbb{C}^{*}})
\]
and, as the $\mathbb{C}^{*}$-action on $S$ contracts $S$ to the
point $\{e\}$, we have that associated filtration on $\Gamma(O_{h,S}[h^{-1}])^{\mathbb{C}^{*}}$
is concentrated in positive degrees, and the associated graded of
this filtration is therefore isomorphic to $\Gamma(O_{h,S})/h\Gamma(O_{h,S})$.
Applying \ref{lem:Basic-CC-lemma} to $O_{h,S}$ we have that
\[
\Gamma(O_{h,S})/h\Gamma(O_{h,S})\tilde{\to}\Gamma(O_{h,S}/h)=\Gamma(O_{S})
\]
On the other hand, from the definition of $U(\mathfrak{g},e)$ as
a Hamiltonian reduction, there clearly exists a map $U(\mathfrak{g},e)\to\Gamma(O_{h,S}[h^{-1}])$
deduced from the map $U(\mathfrak{g})\to\Gamma(O_{h,\mathfrak{g}^{*}}[h^{-1}])$,
which takes an element $g\in\mathfrak{g}(i)$ to $h^{-(i+2)/2}g$
(as discussed below \ref{lem:This-action-preserves}). This
map has image contained in the $\mathbb{C}^{*}$-fixed locus, and
thus the map $U(\mathfrak{g},e)\to\Gamma(O_{h,S}[h^{-1}])$ actually
lands in $\Gamma(O_{h,S}[h^{-1}])^{\mathbb{C}^{*}}$. Taking associated
graded with respect to the induced filtrations on each side, we obtain
a map 
\[
\text{gr}(U(\mathfrak{g},e))\to\Gamma(O_{S})
\]
which, from the very construction of $U(\mathfrak{g},e)$, must be
the identity map. Therefore the map $U(\mathfrak{g},e)\to\Gamma(O_{h,S}[h^{-1}])^{\mathbb{C}^{*}}$
is an isomorphism, and the result follows. 
\end{proof}
Now that this has been done, we aim describe our localization theorem
for the finite $W$-algebras, and give several proofs of it. We first
need to recall some of the relevant geometry. Let $e\in N$. Letting
$S$ denote the Slodowy slice as above, we have the singular algebraic
variety $S\cap N:=S_{e}$, where $N$ is the nilpotent cone in $\mathfrak{g}$.
Then, if $\mu:T^{*}X\to N$ is the springer resolution, we have that
$\tilde{S}_{e}:=\mu^{-1}(S_{e})\to S_{e}$ is also a resolution of
singularities. In particular, $e$ is a regular value for the map
$\mu$, a crucial fact for our considerations (c.f. \cite{key-16},
chapter 2, for a detailed proof). 

We shall realize this resolution as a Hamiltonian reduction of the
left action of the group $M_{l}$ on the space $T^{*}X$. We think
of $\mu:T^{*}X\to\mathfrak{g}^{*}$ (using the original definition
of $\mu$ as a moment map), and we note that the moment map for $M_{l}$,
$\mu^{'}$, is given by the composition $T^{*}X\to\mathfrak{g}^{*}\to\mathfrak{m}_{l}^{*}$,
where the second map is the restriction of functions. We consider
$\chi\in\mathfrak{m}_{l}^{*}$. Then, using the alternate description
of $T^{*}X$ as an incidence variety, we have $(\mu^{'})^{-1}(\chi)=\{(x,b)\in\mathfrak{g}\times X|x\in b,x\in(\mathfrak{m}_{l}^{\perp}+\chi)\cap N\}$,
where $\mathfrak{m}_{l}^{\perp}$ denotes the annihilator of $\mathfrak{m}_{l}$
in $\mathfrak{g}$ under the Killing form (so this corresponds to
those functionals in $\mathfrak{g}^{*}$ which die on $\mathfrak{m}_{l}$,
the kernel of the restriction map $\mathfrak{g}^{*}\to\mathfrak{m}_{l}^{*}$).
But now, according to \cite{key-3}, we have an isomorphism $M_{l}\times S\to\mathfrak{m}_{l}^{\perp}+\chi$
which is simply the adjoint action $(m,s)\to ad(m)(s)$. Therefore,
under the same map, we have an isomorphism $M_{l}\times(S_{e})\to(\mathfrak{m}_{l}^{\perp}+\chi)\cap N$. 

Now, the action of $M_{l}$ on $T^{*}X$ (thinking of $T^{*}X$ as
an incidence variety), is given as follows: $m(x,b)=(ad(m)(x),mbm^{-1})$.
Further, we write any element of $(\mu^{'})^{-1}(\chi)$ uniquely
as $(ad(m)(y),b)$ (with $y\in S_{e})$, and therefore we have a map
$(\mu^{'})^{-1}(\chi)\to\tilde{S}_{e}$, $(ad(m)(y),b)\to(y,m^{-1}bm)$.
We see immediately that this map is in fact a principal $M_{l}$-bundle,
with $M_{l}\times\tilde{S}_{e}\tilde{\to}(\mu^{'})^{-1}(\chi)$ via
$(m,(y,b))\to(ad(m)(y),mbm^{-1})$. Therefore we have identified $\tilde{S}_{e}$
as a Hamiltonian reduction, and, therefore, a symplectic variety.
We note that by the results in \cite{key-12} (see also \cite{key-16}),
the moment map $\tilde{S}_{e}\to S_{e}$ is a resolution of singularities,
and the base variety $S_{e}$ is normal. 

The next step is to consider the Hamiltonian reduction of differential
operators. We have the action of $M_{l}$ on $T^{*}X$, and we consider
modules with twist $\chi$ as above. As the variety $\tilde{S}_{e}$
is smooth symplectic and $\mu^{-1}(\chi)$ is a a principle $M_{l}$
-bundle over it, we see that the conditions of \ref{lem:Good-Conditions-For-T^*}
(and therefore \ref{prop:Reduction-equivalence}) are satisfied.
We can thus apply Hamiltonian reduction to $D_{h}(\lambda)$ we obtain
a quantization of $\tilde{S}$ which we call $D_{h}(\lambda,\chi)$. 
\begin{thm}
\label{thm:Localization-For-W}For $\lambda$ anti-dominant, we have
equivalences of categories 
\[
Mod^{\mathbb{C}^{*},coh}(D_{h}(\lambda,\chi)\tilde{\to}Mod^{f.g.}(U(\mathfrak{g},e)_{\lambda})
\]
\[
Mod^{\mathbb{C}^{*},qc}(D_{h}(\lambda,\chi)\tilde{\to}Mod(U(\mathfrak{g},e)_{\lambda})
\]
given by $\Gamma^{\mathbb{C}^{*}}$. 
\end{thm}

\begin{proof}
By the results above, the category $Mod^{f.g.}(U(\mathfrak{g},e)_{\lambda})$
is equivalent to\\
$Mod_{\chi}^{M_{l},f.g.}(U(\mathfrak{g})_{\lambda})$, which, by applying
the localization theorem \ref{thm:Localization-with-actions},
is equivalent to $Mod_{\chi}^{M_{l}\rtimes\mathbb{C}^{*},coh}(D_{h}(\lambda))$.
Applying Hamiltonian reduction (i.e., \ref{prop:Reduction-equivalence}),
we obtain that this category is equivalent to $Mod^{\mathbb{C}^{*},coh}(D_{h}(\lambda,\chi)$.
The first result follows, and the second then follows by taking ind
categories on both sides. 
\end{proof}
In the final part of this chapter, we're going to give another proof
of this theorem, which has a more geometric flavor. In particular,
we shall show that the analogues of several key results in the Beilinson-Bernstein
theory also hold in the $W$-algebra context, leading to a direct
proof of the result. 

We start by considering global sections. We define the algebra $U_{h}(\mathfrak{g},e)(0):=\Gamma(O_{h,S})$,
where $O_{h,S}$ is the quantization of $S$ defined above via Hamiltonian
reduction. Let $U_{h}(\mathfrak{g},e)=U_{h}(\mathfrak{g},e)(0)[h^{-1}]$.
In the course of proving \ref{prop:Ham-reduction-on-g}, we
showed that $U_{h}(\mathfrak{g},e)(0)/h\tilde{\to}\Gamma(O_{S})$
and that $U_{h}(\mathfrak{g},e)^{\mathbb{C}^{*}}=U(\mathfrak{g},e)$. 

Further, for any character $\lambda$ of $\mathfrak{b}$ as considered
in the previous chapter, we had the ideals $J_{\lambda}\subseteq U_{h}(\mathfrak{g})$,
which had the property that $J_{\lambda}(0)/hJ_{\lambda}(0)\tilde{=}I(N)$
($N$ as usual is the nilpotent cone). So we can consider the image
of $J_{\lambda}$ in $U_{h}(\mathfrak{g},e)$, called $B_{\lambda}$,
and we see that $B_{\lambda}(0)/hB_{\lambda}(0)\tilde{=}I(S_{e})$
(this is implied by the fact that $M_{l}\times(S\cap N)\tilde{\to}(\mathfrak{m}_{l}^{\perp}+\chi)\cap N$). 

Now, we have a map $\Psi_{\lambda}:U_{h}(\mathfrak{g},e)\to\Gamma(\tilde{S}_{e},D_{h}(\lambda,\chi))$,
which results from the map $\Phi_{\lambda}$ as both sides are defined
by Hamiltonian reduction. We are now in a situation completely parallel
to that of \ref{lem:Phi-is-iso}, i.e., the technical criterion
of \ref{lem:Basic-CC-lemma} applies on $\tilde{S}_{e}$, and
so we can conclude
\begin{lem}
$\Psi_{\lambda}$ is surjective, and $ker(\Psi_{\lambda})=B_{\lambda}$. 
\end{lem}

We now give the necessary modifications of the proof of \ref{thm:Loc-on-Flag}
so that we may obtain another proof of \ref{thm:Localization-For-W}.
Recall that we have equivalences of categories $Mod^{?}(D_{h}(\lambda))\tilde{\to}Mod^{?}(D_{h}(\lambda+\psi))$
(where $?$ stands for coherent or quasicoherent). These equivalences
were obtained by first lifting an $M(0)\in Mod^{?}(D_{h}(\lambda))$
to an element of $Mod_{\lambda}^{?,B}(D_{h}(T^{*}G)$, then twisting
upstairs, and then pushing back down. Since the $B$-action we're
considering on $T^{*}G$ is on the right, if we consider categories
of the form $Mod_{\chi}^{?,M_{l}}(D_{h}(\lambda))$, then as the $M_{l}$
action is on the left, this process gives us equivalences 
\[
Mod_{\chi}^{?,M_{l}}(D_{h}(\lambda))\tilde{\to}Mod_{\chi}^{?,M_{l}}(D_{h}(\lambda+\psi))
\]
 Furthermore, we also have equivalences 
\[
Mod_{\chi}^{?,M_{l}}(D_{h}(\lambda)\tilde{\to}Mod^{?}(D_{h}(\lambda,\chi)
\]
 Combining these, we get equivalences 
\[
Mod^{?}(D_{h}(\lambda,\chi)\tilde{\to}Mod^{?}(D_{h}(\lambda+\psi,\chi)
\]

As before, we call the resulting functor $F_{\psi}$, and we can give
a description of how it acts: if we let $M(0)\in Mod^{?}(D_{h}(\lambda,\chi)(0))$,
then it follows from the definitions that 
\[
F_{\psi}(M(0))/hF_{\psi}(M(0))\tilde{=}M(0)/hM(0)\otimes p_{*}(O_{\psi}|_{\mu^{-1}(\chi)})^{M_{l}}
\]

But we have that the space $\tilde{S}_{e}$ is a subscheme of $\tilde{N}$
as well as a Hamiltonian reduction. So, if we consider $O_{\psi}|_{\tilde{S}_{e}}$,
then we have that $p^{*}(O_{\psi}|_{\tilde{S}_{e}})\tilde{=}O_{\psi}|_{\mu^{-1}(\chi)}$,
since $O_{\psi}$ is an $M_{l}$-equivariant bundle and $\mu^{-1}(\chi)\tilde{=}M_{l}\times\tilde{S}_{e}$
(as explained above). So now it follows that $p_{*}(O_{\psi}|_{\mu^{-1}(\chi)})^{M_{l}}\tilde{=}O_{\psi}|_{\tilde{S}_{e}}$. 

The next step is to define the analogue of the functors $G_{\nu}$.
This is done in the natural way: for any $M(0)\in Mod^{?}(D_{h}(\lambda,\chi))$,
we can consider the pullback to a module $N(0)\in Mod_{\chi}^{?,M_{l}}(D_{h}(\lambda)(0))$;
we then apply the functor $G_{\nu}$ to obtain $G_{\nu}(M(0))\in Mod_{\chi}^{?,M_{l}}(D_{h}(\mathfrak{h})(0))$.
In order to take the reduction of such a module, we first note that
the reduction functor of \ref{prop:Reduction-equivalence} (which
we only defined for an object of a particular $Mod(D_{h}(\lambda)(0))$)
can actually be defined as $S\to p_{*}(\hat{S})$ where $\hat{S}$
is the subsheaf consisting of local sections $m$ of $S$ such that
$\xi_{x}\cdot m=\chi(x)m$ for all $x\in\mathfrak{m}_{l}$. So, this
functor actually makes sense for any object in $Mod(D_{h}(\mathfrak{h})$.
We again call the resulting functor $G_{\nu}$. We note that $G_{\nu}(M(0))$
is now just a sheaf of abelian groups. However, as the functor $G_{\nu}$
(for sheaves on $T^{*}X$) simply amounted to taking a finite direct
sum of copies of the input sheaf, we conclude that the same is true
of the new $G_{\nu}$. So, we obtain: 
\[
G_{\nu}(M(0))\tilde{=}M(0)\otimes L(\nu)|_{\tilde{S}_{e}},
\]
 Finally, since this reduction procedure is (at the very least) an
additive functor on sheaves of abelian groups, we can conclude from
\ref{lem:basic-weight-lemma}
\begin{lem}
Let $\lambda$ be an anti-dominant weight, and $\mu$ a dominant integral
weight, and let $M\in Mod^{qcoh}(D_{h}(\lambda,\chi))$. Then \textup{$G_{\mu}F_{-\mu}(M)$}
has $M$ as a direct summand. Further, let $w_{0}$ denote the longest
element of the Weyl group. Then the sheaf $F_{-\mu}(M)$ is a direct
summand of $G_{-w_{0}\mu}(M)\tilde{=}M\otimes L(-w_{0}\mu)$. 
\end{lem}

Now, given an anti-dominant weight $\psi$, $O_{\psi}$ is an ample
line bundle on $\tilde{N}$ (with respect to the base scheme $N$).
Therefore, its restriction to $\tilde{S}_{e}$ is ample with respect
to $S_{e}$. So we see that we have all the ingredients that gave
us the proof of \ref{thm:Loc-on-Flag} (i.e., the proof that
we gave followed formally from the above lemmas and general facts
about quantized sheaves of algebras). Thus, we can conclude: 
\begin{thm}
Let $\lambda$ be an anti-dominant weight. Then 
\[
\Gamma:Mod^{qc}(D_{h}(\lambda,\chi))\to Mod^{qc}(U_{h}(\mathfrak{g},e)/B_{\lambda})
\]
is an equivalence of categories. Further, $\Gamma$ takes coherent
$D_{h}(\lambda,\chi)$ modules to finitely generated $U_{h}(\mathfrak{g},e)$
modules, and we have that 
\[
\Gamma:Mod^{coh}(D_{h}(\lambda,\chi))\to Mod^{coh}(U_{h}(\mathfrak{g},e)/B_{\lambda})
\]
is an equivalence of categories as well. 
\end{thm}

Of course, this theorem is not really what we want. To put things
in their final form, we need to consider a $\mathbb{C}^{*}$-action
on the category of modules. Fortunately, we have that the Hamiltonian
reduction procedure respects the Gan-Ginzburg $\mathbb{C}^{*}$-action
on $D_{h}(\lambda)$: the ideal $I_{\chi}$ is clearly $\mathbb{C}^{*}$-invariant,
and the process of taking $M_{l}$-invariants respects the $\mathbb{C}^{*}$-action
because of the commutation relations between $M_{l}$ and $\mathbb{C}^{*}$.
Therefore, $D_{h}(\lambda,\chi)$ is $\mathbb{C}^{*}$-equivariant
with respect to the $\mathbb{C}^{*}$ action on $\tilde{S}_{e}$. 

This will allow us to identify the $\mathbb{C}^{*}$- invariant global
sections of $D_{h}(\lambda,\chi)$ as follows: we have seen above
that $U_{h}(\mathfrak{g},e)$ carries a natural $\mathbb{C}^{*}$
action with respect to which 
\[
(U_{h}(\mathfrak{g},e))^{\mathbb{C}^{*}}=U(\mathfrak{g},e)
\]
We also concluded above that $J_{\lambda}\cap U_{h}(\mathfrak{g})^{\mathbb{C}^{*}}$
was the ideal generated by the classical ideal $I_{\lambda}$. So
it follows that $B_{\lambda}\cap(U_{h}(\mathfrak{g},e))^{\mathbb{C}^{*}}$
is the image of this ideal in $U(\mathfrak{g},e)$. But we have an
identification of the center of $U(\mathfrak{g},e)$ with the center
of $U(\mathfrak{g})$ (the natural map $Z(\mathfrak{g})\to Z(\mathfrak{g},e)$
is an isomorphism, see \cite{key-14} section 5, footnote 2). So in
fact we can conclude that $\Gamma(\tilde{S}_{e},D_{h}(\lambda,\chi)^{\mathbb{C}^{*}}\tilde{=}U(\mathfrak{g},e)/I_{\lambda}:=U(\mathfrak{g},e)_{\lambda}$.
Thus, after taking into account the $\mathbb{C}^{*}$-actions, the
previous theorem immediately implies another proof of \ref{thm:Localization-For-W}.


\begin{thebibliography}{BDMN}
\bibitem[BB]{key-15}A. Beilinson and J. Bernstein, \emph{Localisation
de g-modules}, C.R. Acad. Sci. Paris Sr. I Math. 292 (1981), no. 1.
15-18. 

\bibitem[BDMN]{key-25}G. Bellamy, C. Dodd, K. McGerty, and T. Nevins,
\emph{Categorical Cell Decomposition of Quantized Symplectic Algebraic
Varieties, }Geom. Topol., 21(5): 2601-2681, 2017.\emph{ }

\bibitem[BK1]{key-5}R. Bezrukavnikov and D. Kaledin, \emph{Fedesov
Quantization in algebraic Context}, Moscow Math J. 4 (2004), 559-592. 

\bibitem[BK2]{key-6}R. Bezrukavnikov and D. Kaledin, \emph{McKay
Equivalence for Symplectic quotient Singularities}, Proc. of the Steklov
Inst. of Math., 246 (2004), 13-33.

\bibitem[CG]{key-2}N. Chriss and V. Ginzburg, \emph{Representation
Theory and Complex Geometry}, Birkhauser Boston, 1997. 

\bibitem[DR]{key-20}G. Dhillon and S. Raskin, \emph{Localization
for affine W-algebras}. Adv. Math. 413 (2023), Paper No. 108837, 51
pp

\bibitem[D]{key-22}C. Dodd, \emph{Injectivity of the cycle map for
finite-dimensional W-algebras}. Int. Math. Res. Not. IMRN 2014, no.
19, 5398\textendash 5436.

\bibitem[GG]{key-3}W.L. Gan and V. Ginzburg, \emph{Quantization of
Slodowy Slices}, Int. Math. Res. Not. 5 (2002), 243-255.

\bibitem[G1]{key-16} V. Ginzburg, \emph{Harish-Chandra Bimodules
for Quantized Slodowy Slices}, Represent. Theory 13 (2009), 236\textendash 271.

\bibitem[G2]{key-4} V. Ginzburg, \emph{On Primitive Ideals}, Selecta
Math, new series, 9(2003), 379-407. 

\bibitem[H]{key-10}R. Hartshorne, \emph{Algebraic Geometry, }Graduate
Texts in Mathematics, 52, Springer, 1977.

\bibitem[HTT]{key-1}R. Hotta, K. Takeuchi, and T. Tanisaki,\emph{
D-modules, Perverse Sheaves, and Representation Theory}, Progress
in Mathematics, 236, Birkhauser Boston, 2008.

\bibitem[KR]{key-7} M. Kashiwara and R. Rouquier, \emph{Microlocalization
of Rational Cherednik Algebras}, Duke Math. J. 144 (2008), no. 3,
525\textendash 573.

\bibitem[KS]{key-8}M.Kashiwara and P. Shapira, \emph{Deformation
quantization modules}. Astérisque No. 345 (2012), xii+147 pp.

\bibitem[Kn]{key-9} A. Knapp, \emph{Lie Groups Beyond an Introduction,
Second Edition, }Progress in Mathematics, 140, Birkhauser Boston,
2002. 

\bibitem[M]{key-11} D. Milicic, \emph{Localization and Representation
Theory of Reductive Lie Groups}, available at http://www.math.utah.edu/\textasciitilde{}milicic. 

\bibitem[Pr1]{key-13} A. Premet, \emph{Special Transverse Slices
and Their Enveloping Algebras, }Adv. Math. 170(2002), 1-55.

\bibitem[Pr2]{key-14} A. Premet, \emph{Enveloping Algebras of Slodowy
Slices and the Joseph ideal. }J. Eur. Math. soc, to appear, arXiv:math.RT/0504343.\emph{ }

\bibitem[Sl]{key-12} P. Slodowy, \emph{Simple Singularities and Simple
Algebraic Groups, }Lecture Notes in Mathematics, 815, Springer, Berlin
(1980). 
\end{thebibliography}
\end{document}